\documentclass[a4paper,10pt]{book}
\usepackage[latin1]{inputenc}
\usepackage[T1]{fontenc}
\usepackage{amsmath}
\usepackage{amsfonts}
\usepackage{amstext}
\usepackage{amsthm}
\usepackage[all]{xy}
\usepackage{geometry}
\usepackage{graphicx}

\swapnumbers
\theoremstyle{plain}

\newtheorem{theorem}{Theorem}[section]
\newtheorem{lemma}[theorem]{Lemma}
\newtheorem{defi}[theorem]{Definition}

\newtheorem{cor}[theorem]{Corollary}
\newtheorem{example}[theorem]{Example}

\DeclareMathOperator{\im}{im}
\DeclareMathOperator{\ind}{ind}

\DeclareMathOperator{\sfl}{sf}

\DeclareMathOperator{\sgn}{sgn}
\DeclareMathOperator{\gra}{graph}

\DeclareMathOperator{\codim}{codim}
\DeclareMathOperator{\calkin}{Cal}

\title{Fredholm Operators and Spectral Flow}
\author{Nils Waterstraat}
\date{}

\begin{document}

\maketitle

\tableofcontents

\chapter*{Introduction}
\textit{Fredholm operators} are one of the most important classes of linear operators in mathematics. They were introduced around 1900 in the study of integral operators and by definition they share many properties with linear operators between finite dimensional spaces. They appear naturally in \textbf{global analysis} which is a branch of pure mathematics concerned with the global and topological properties of systems of differential equations on manifolds. One of the basic important facts says that every linear elliptic differential operator acting on sections of a vector bundle over a closed manifold induces a Fredholm operator on a suitable Banach space completion of bundle sections. Every Fredholm operator has an integer-valued index, which is invariant under deformations of the operator, and the most fundamental theorem in global analysis is the \textit{Atiyah-Singer index theorem} \cite{ASThm} which gives an explicit formula for the Fredholm index of an elliptic operator on a closed manifold in terms of topological data. An important special case are selfadjoint elliptic operators, which naturally appear in geometry and physics but for which, however, the Fredholm index vanishes. Atiyah, Patodi and Singer introduced in \cite{AtiyahPatodi} the \textit{spectral flow} as an integer-valued homotopy invariant for (closed) paths of selfadjoint Fredholm operators. Roughly speaking, the spectral flow is the number of eigenvalues which pass through zero in the positive direction from the start of the path to its end. Atiyah, Patodi and Singer proved in \cite{AtiyahPatodi} that the spectral flow of a closed path of selfadjoint elliptic differential operators on a closed manifold can be computed by a topological index in essentially the same way as the Fredholm index in the previous index theorem \cite{ASThm} of Atiyah and Singer. In what follows, we denote the spectral flow of a path $\mathcal{A}=\{\mathcal{A}_\lambda\}_{\lambda\in S^1}$ of selfadjoint Fredholm operators by $\sfl(\mathcal{A})$.\\
The spectral flow has been defined in several different but equivalent ways and it has become a well known and widely used integer-valued homotopy invariant for generally non-closed paths $\mathcal{A}=\{\mathcal{A}_\lambda\}_{\lambda\in I}$ of selfadjoint Fredholm operators (cf. \cite{Phillips}, \cite{UnbSpecFlow}). For example, lots of efforts have been made to compute the spectral flow for paths $\mathcal{A}$ of operators induced by boundary value problems for first order selfadjoint elliptic operators on manifolds with boundary. Several spectral flow formulas have been found in this setting expressing $\sfl(\mathcal{A})$ in terms of various invariants (cf. e.g. \cite{LeschWoj} among many other references). The spectral flow has also been used in \textbf{symplectic analysis}, where, e.g., Floer introduced it as a substitute for the Morse index in order to define a grading of his celebrated homology groups \cite{Floer}. It was thoroughly studied in this setting by Robbin and Salamon in \cite{Robbin-Salamon}, and it was also used before by Salamon and Zehnder in \cite{Salamon-Zehnder}.\\
The aim of these notes is an essentially self-contained introduction to the spectral flow for paths of (generally unbounded) selfadjoint Fredholm operators. We begin by recapitulating well known theory about bounded and unbounded operators in the first two sections following \cite{GohbergClasses}, \cite{Goldberg}, \cite{Weidmann} and \cite{Rudin}, where we particularly focus on spectral theory. The third section is devoted to the gap-topology on the space of all closed operators on a given Hilbert space, which we need in order to deal with continuous paths of operators. In the fourth section we construct the spectral flow and discuss some of its properties following \cite{UnbSpecFlow}, \cite{LeschSpecFlowUniqu} and \cite{Robbin-Salamon}. The final section is devoted to a simple example and some hints for further reading.\\
\vspace{0.1cm}

\noindent
This manuscript are extended lecture notes of a PhD course that the author gave at the Universit\`a degli studi di Torino in Italy in spring 2013. We are grateful to Anna Capietto and Alessandro Portaluri for inviting us to give these lectures, and to the audience for several valuable questions.


\chapter{Linear Operators}

\section{Bounded Operators and Subspaces}
Let $E$ and $F$ be non-trivial complex Banach spaces. We denote throughout by $I_E,I_F$ the identity operators on $E$ and $F$, respectively. Recall that a linear operator $A:E\rightarrow F$ is \textit{bounded} if there exists a constant $c\geq 0$ such that

\begin{align}\label{boundedop}
\|Au\|\leq c\|u\|,\quad u\in E.  
\end{align}
The smallest possible bound in \eqref{boundedop} is the \textit{norm} $\|A\|$ of $A$ and it is given by

\begin{align}\label{normequa}
\|A\|=\sup_{u\neq 0}\frac{\|Au\|}{\|u\|}=\sup_{\|u\|\leq 1}{\|Au\|}=\sup_{\|u\|=1}\|Au\|.
\end{align}
It is a simple exercise for boring train trips to check the equalities in \eqref{normequa} as well as the following lemmata.

\begin{lemma}
The following assertions are equivalent:
\begin{itemize}
\item $A$ is continuous,
\item $A$ is continuous at some $u\in E$,
\item $A$ is continuous at $0\in E$,
\item $A$ is bounded.
\end{itemize}
\end{lemma}
\noindent
In what follows, we denote by $\mathcal{L}(E,F)$ the set of all bounded operators $A:E\rightarrow F$ which is a normed linear space with respect to the operator norm \eqref{normequa}. 

\begin{lemma}
$\mathcal{L}(E,F)$ is a Banach space.
\end{lemma}
\noindent
The proof of the previous lemma only uses that $F$ is complete and hence the dual space $N^\ast:=\mathcal{L}(N,\mathbb{C})$ of any normed linear space $N$ is a Banach space.\\
The following important result in functional analysis is called the \textit{open mapping theorem}, and it is frequently applied to prove the continuity of inverses of operators.

\begin{theorem}\label{openmapping}
If $A\in\mathcal{L}(E,F)$ is surjective, then $A$ is an open map, i.e., $A(U)\subset F$ is open for any open subset $U\subset E$.
\end{theorem}
\noindent
In contrast, \textit{compact operators} are never boundedly invertible on infinite dimensional spaces:
\begin{defi}
A linear operator $A:E\rightarrow F$ is compact if $A(U)\subset F$ is relatively compact for any bounded subset $U\subset E$.
\end{defi}
\noindent
We denote by $\mathcal{K}(E,F)$ the set of all compact operators $A:E\rightarrow F$.

\begin{lemma}
Every compact operator is bounded; i.e., $\mathcal{K}(E,F)\subset\mathcal{L}(E,F)$. Moreover, $\mathcal{K}(E,F)$ is closed and products of compact and bounded operators are compact.
\end{lemma}
\noindent 
Now we assume that $U,V$ are subspaces of $E$ such that $E=U\oplus V$ algebraically. In this case we have a unique projection $P:E\rightarrow E$ onto $U$ with respect to this decomposition and we can ask about the boundedness of $P$.

\begin{lemma}\label{APPFA-lemma-proj}
The projection $P:E\rightarrow E$ is bounded if and only if $U$ and $V$ are closed.
\end{lemma}
\noindent
A strictly related problem concerns the following definition.

\begin{defi}\label{APPFA-defi-complemented}
A closed subspace $U$ of a Banach space $E$ is called \textit{complemented} if there exists a closed subspace $V$ of $E$ such that $E=U\oplus V$.
\end{defi}
\noindent
A well known example of a non-complemented subspace of a Banach space is given by

\begin{align*}
\{\{x_n\}_{n\in\mathbb{N}}\subset\mathbb{C}:\lim_{n\rightarrow\infty} x_n=0\}
\end{align*}
which is a closed subspace of

\begin{align*}
l^\infty=\{\{x_n\}_{n\in\mathbb{N}}\subset\mathbb{C}:\sup_{n\in\mathbb{N}}|x_n|<\infty\},\quad \|\{x_n\}_{n\in\mathbb{N}}\|=\sup_{n\in\mathbb{N}}|x_n|.
\end{align*}
However, we have the following two positive results.

\begin{theorem}\label{APPFA-theorem-Hilbcomplemented}
Any closed subspace $U$ of a Hilbert space $H$ is complemented.
\end{theorem}
\noindent
Note that the latter result is just the well known theorem on orthogonal projections and decompositions $H=U\oplus U^\perp$ for any closed subspace $U\subset H$.

\begin{lemma}\label{APPFA-lemma-finitecomplemented}
Any subspace $U$ of $E$ of finite dimension is complemented.
\end{lemma}
\noindent
Lemma \ref{APPFA-lemma-finitecomplemented} suggests the question if the same result holds if we require the codimension of $U$ instead of its dimension to be finite. That is, we assume to have a decomposition $E=U\oplus V$ where $V$ is of finite dimension. Since finite dimensional spaces are closed we obtain immediately that any \textit{closed} subspace of $E$ of finite codimension is complemented. However, by the following result, we cannot omit the closedness assumption.

\begin{lemma}
In any Banach space $E$ of infinite dimension, there exists a one-codimensional subspace $U$ which is not closed.
\end{lemma}

\begin{proof}
Since $E$ is of infinite dimension, there exists an unbounded linear functional $f:E\rightarrow\mathbb{R}$. Then $U:=\ker(f)$ has codimension $1$ and it is not closed. For, assume that $U$ is closed. Since $V$ is of finite dimension, we have $E=U\oplus V$, where $U$ and $V$ are closed. But $f\mid_U\equiv 0$ and $f\mid_V$ are continuous and consequently $f:E\rightarrow\mathbb{C}$ would be continuous.  
\end{proof}
\noindent
The following remarkable result implies that a finite codimensional subspace $U\subset F$ is complemented if it is the image of a bounded linear operator. 

\begin{lemma}\label{codimclosed}
If $A\in\mathcal{L}(E,F)$ is a bounded linear operator and $V\subset F$ is a closed subspace such that $F=\im(A)\oplus V$, then $\im(A)$ is closed.
\end{lemma}

\begin{proof}
Let $P:E\rightarrow E/\ker(A)$ denote the projection onto the quotient space. Then $A:E\rightarrow F$ factorises as

\begin{align*}
\xymatrix{
E\ar[dr]_P\ar[rr]^A& &\im(A)\subset F\\
&E/\ker(A)\ar[ur]_S&
}
\end{align*}
where $S:E/\ker(A)\rightarrow \im(A)$ is the bounded and bijective map defined by $S[u]=Au$, $u\in[u]\in X/\ker(A)$. We now define
\[\overline{S}:E/\ker(A)\times V\rightarrow F,\quad \overline{S}([u],v)=S[u]+w\]
and note that $E/\ker(A)\times V$ is a Banach space as $V$ and $\ker(A)$ are closed. Now $\overline{S}$ is a bijective bounded linear operator and hence $\overline{S}^{-1}$ is bounded by Theorem \ref{openmapping}. As $E/\ker(A)\times\{0\}$ is closed in $E/\ker(A)\times V$, we obtain that 

\[\im(A)=S(E/\ker(A))=\overline{S}(E/\ker(A)\times\{0\})\subset F\]
is closed.
\end{proof}






\section{Closed Operators}

As before, we let $E$ and $F$ be non-trivial complex Banach spaces. A \textit{linear operator} acting between $E$ and $F$ is a linear map $T:\mathcal{D}(T)\rightarrow F$, where $\mathcal{D}(T)$ is a linear subspace of $E$. We usually write
\[T:\mathcal{D}(T)\subset E\rightarrow F\]
in order to emphasise the space $E$. The operator $T$ is called \textit{densely defined} if $\mathcal{D}(T)$ is dense in $E$ and it is called \textit{bounded} if there exists $c>0$ such that
\[\|Tu\|\leq c\|u\|,\quad u\in\mathcal{D}(T).\]
Given linear operators $T:\mathcal{D}(T)\subset E\rightarrow F$, $S:\mathcal{D}(S)\subset E\rightarrow F$ $R:\mathcal{D}(R)\subset F\rightarrow G$  and $\alpha\in\mathbb{C}$, we define

\begin{itemize}
	\item $\alpha\,T:\mathcal{D}(\alpha\,T)\subset E\rightarrow F$ by $(\alpha\, T)u=\alpha\,(Tu)$, $u\in\mathcal{D}(\alpha\,T)=\mathcal{D}(T)$,
	\item $T+S:\mathcal{D}(T+S)\subset E\rightarrow F$ by $(T+S)u=Tu+Su$, $u\in\mathcal{D}(T+S)=\mathcal{D}(T)\cap\mathcal{D}(S)$,
	\item $RS:\mathcal{D}(RS)\subset E\rightarrow G$ by $(RS)u=R(Su)$,
	\[u\in\mathcal{D}(RS)=S^{-1}(\mathcal{D}(R))=\{u\in E:\, u\in\mathcal{D}(S),\,Su\in\mathcal{D}(R)\}.\]
\end{itemize}
We write $T\subset S$ if $\mathcal{D}(T)\subset\mathcal{D}(S)$ and $Tu=Su$ for all $u\in\mathcal{D}(T)$. Moreover, we set $T=S$ if $T\subset S$ and $S\subset T$. Finally, we denote by $0$ the operator acting by $u\mapsto 0$ on the domain $\mathcal{D}(0)=E$.\\
For three operators $R,S,T$ acting between $E$ and $F$, one easily verifies that

\begin{itemize}
\item $0\,T\subset 0$, $0+T=T+0=T$,
\item $(R+S)+T=R+(S+T)$,
\item $S+T=T+S$,
\item $(S+T)-T\subset S$.
\end{itemize}
Moreover, if $T,T_1,T_2$ are operators from $E$ to $F$, $S,S_1,S_2$ map from $F$ to $G$ and $R$ has its range in $E$, then

\begin{itemize}
\item $(ST)R=S(TR)$,
\item $(\alpha\, S)T=S(\alpha\, T)=\alpha(ST)$ if $\alpha\neq 0$,
\item $(0\,S)T=0(ST)\subset T(0\,S)$,
\item $(S_1+S_2)T=S_1T+S_2T$,
\item $S(T_1+T_2)\supset ST_1+ST_2$,
\item $I_FT=TI_E=T$.
\end{itemize}

\begin{defi}\label{defi-invertible}
We call an operator $T:\mathcal{D}(T)\subset E\rightarrow F$ \textit{invertible}, if $T$ maps $\mathcal{D}(T)$ bijectively onto $F$. We say that $T$ has \textit{a bounded inverse}, if $T$ is invertible and $T^{-1}:F\rightarrow \mathcal{D}(T)\subset E$ is bounded.
\end{defi}
\noindent
Let us point out that there are various non-equivalent definitions of invertibility in the literature. For example, in \cite{Kato} an operator is called invertible if it is injective. On the contrary, in \cite{Hislop} an operator is said to be invertible if it has a bounded inverse according to Definition \ref{defi-invertible}.\\  
For an operator $T:\mathcal{D}(T)\subset E\rightarrow F$, we denote by

\[\gra(T)=\{(u,Tu)\in E\times F:\, u\in\mathcal{D}(T)\}\subset E\times F\]
the \textit{graph} of $T$, which is a linear subspace of the Banach space $E\times F$. Clearly, $T\subset S$ is equivalent to $\gra(T)\subset\gra(S)$.

\begin{defi}
The operator $T:\mathcal{D}(T)\subset E\rightarrow F$ is called \textit{closed} if $\gra(T)$ is a closed subspace of $E\times F$.
\end{defi}
\noindent
Note that $T$ is closed if and only if, for every sequence $\{u_n\}_{n\in\mathbb{N}}\subset\mathcal{D}(T)$ such that $(u_n,Tu_n)\rightarrow(u,v)\in E\times F$ for some $u\in E$ and $v\in F$ we have $u\in\mathcal{D}(T)$ and $Tu=v$.

\begin{example}\label{example-closed}
We consider the operator

\[T:C^1[0,1]\subset C[0,1]\rightarrow C[0,1],\quad Tu=u',\]
which is not bounded as $\|Tu_n\|_\infty=n=n\|u_n\|_\infty$ for $u_n(t)=t^n$, $n\in\mathbb{N}$. We claim that $T$ is closed. Let $\{u_n\}_{n\in\mathbb{N}}\subset C^1[0,1]$ be a sequence and $v\in C[0,1]$ such that $u_n\rightarrow u$ and $Tu_n=u'_n\rightarrow v$ uniformly. Then

\begin{align*}
\int^t_0{v(s)\,ds}=\int^t_0{\lim_{n\rightarrow 0}u'_n(s)\,ds}=\lim_{n\rightarrow\infty}\int^t_0{u'_n(s)\,ds}=u(t)-u(0),\,\, t\in[0,1],
\end{align*}
and hence
\[u(t)=u(0)+\int^t_0{v(s)\,ds},\quad t\in[0,1],\]
where $v\in C[0,1]$. Consequently, $u\in C^1[0,1]$, $Tu=v$ and so $T$ is closed.
\end{example}
\noindent
In what follows we denote the set of all closed operators acting between $E$ and $F$ by $\mathcal{C}(E,F)$. In contrast to $\mathcal{L}(E,F)$, the set $\mathcal{C}(E,F)$ is not a linear space:

\begin{example}
Assume that $D\subset E$ is dense and strictly contained in $E$ and let $T:\mathcal{D}(T)\subset E\rightarrow F$ be a closed operator defined on $\mathcal{D}(T)=D$ (e.g., $E=C[0,1]$, $Tu=u'$ on $\mathcal{D}(T)=C^1[0,1]$ as in Example \ref{example-closed}). Then $T-T=0\,T\subset 0$ is not closed because 

\[\gra(T-T)=\gra(0\,T)=D\times\{0\}\]
is not a closed subspace of $E\times F$.
\end{example}

\begin{lemma}\label{boundedclosed}
A bounded operator $T$ is closed if and only if $\mathcal{D}(T)$ is closed in $E$.
\end{lemma}

\begin{proof}
Assume that $T$ is bounded and closed, and let $\{u_n\}_{n\in\mathbb{N}}\subset\mathcal{D}(T)$ be a sequence such that $u_n\rightarrow u$ in $E$. Since $T$ is bounded, $\{Tu_n\}_{n\in\mathbb{N}}$ is a Cauchy sequence in $F$ and hence converges to some $v\in F$. Consequently, $(u_n,Tu_n)\rightarrow (u,v)$, $n\rightarrow\infty$, and as $T$ is closed, we conclude that $u\in\mathcal{D}(T)$ (and $v=Tu$).\\
Conversely, assume that $T$ is bounded and $\mathcal{D}(T)$ is closed. If $\{u_n\}_{n\in\mathbb{N}}$ is a sequence converging to some $u\in E$, then $u\in\mathcal{D}(T)$. As $T$ is bounded, we obtain that $Tu_n\rightarrow Tu$ and hence $T$ is closed. 
\end{proof}
\noindent
We obtain as a consequence of the previous lemma that $\mathcal{L}(E,F)\subset\mathcal{C}(E,F)$.

\begin{lemma}\label{linearprop}
Let $T\in\mathcal{C}(E,F)$ be a closed operator.
\begin{enumerate}
	\item[(i)] $\alpha\,T$ is closed for any $0\neq\alpha\in\mathbb{C}$.
	\item[(ii)] If $B:\mathcal{D}(B)\subset E\rightarrow F$ is bounded and $\mathcal{D}(T)\subset\mathcal{D}(B)$, then $T+B\in\mathcal{C}(E,F)$.
	\item[(iii)] If $T$ is invertible, then $T^{-1}\in\mathcal{C}(F,E)$. 
\end{enumerate}
\end{lemma}

\begin{proof}
For the first assertion we just note that the map
\[E\times F\rightarrow E\times F,\quad (u,v)\mapsto(u,\alpha\,v)\]
is a homeomorphism mapping $\gra(T)$ to $\gra(\alpha\,T)$. In order to show the second assertion let $\{u_n\}_{n\in\mathbb{N}}$ be a sequence in $\mathcal{D}(T+B)=\mathcal{D}(T)$ such that $u_n\rightarrow u\in E$ and assume that $(T+B)u_n=Tu_n+Bu_n\rightarrow v$, $n\rightarrow\infty$. As $B$ is bounded, $\mathcal{D}(B)$ is closed by Lemma \ref{boundedclosed} and $Bu_n\rightarrow Bu$. We obtain $Tu_n\rightarrow v-Bu$. Since $T$ is closed, this implies $u\in\mathcal{D}(T)$ and $Tu=v-Bu$. Consequently, $u\in\mathcal{D}(T+B)$ and $v=(T+B)u$. For the last assertion, we only have to observe that
\[\{(T^{-1}u,u)\in E\times F: u\in F\}=\{(u,Tu)\in E\times F:\,u\in\mathcal{D}(T)\}=\gra(T)\]
and that the left hand side is mapped to $\gra(T^{-1})$ under the homeomorphism $E\times F\rightarrow F\times E$, $(u,v)\mapsto (v,u)$. 
\end{proof}

\begin{example}
Let us consider the operator $T:\mathcal{D}(T)\subset C[0,1]\rightarrow C[0,1]$, $Tu=u'$, where $\mathcal{D}(T)=\{u\in C^1[0,1]:\, u(0)=0\}$. Then $T$ has a bounded inverse given by
\[T^{-1}:C[0,1]\rightarrow C[0,1],\quad (T^{-1}u)(t)=\int^t_0{u(s)\,ds}\]
and hence $T$ is closed by Lemma \ref{linearprop} and Lemma \ref{boundedclosed}.
\end{example} 
\noindent
Let $T:\mathcal{D}(T)\subset E\rightarrow F$ and $S:\mathcal{D}(S)\subset E\rightarrow F$ be two linear operators such that $\mathcal{D}(T)\subset\mathcal{D}(S)$. Then $S$ is said to be $T$-bounded if there exists $a,b\geq 0$ such that

\begin{align}\label{Tbound}
\|Su\|\leq a\|u\|+b\|Tu\|,\quad u\in \mathcal{D}(T).
\end{align}
The infimum of all $b\geq 0$ for which an $a\geq 0$ exists such that \eqref{Tbound} holds is called the $T$-bound of $S$. Note that the $T$-bound is $0$ if $S$ is bounded. The second part of the previous lemma can be improved as follows.

\begin{theorem}
Let $T:\mathcal{D}(T)\subset E\rightarrow F$ and $S:\mathcal{D}(S)\subset E\rightarrow F$ be two linear operators such that $S$ is $T$-bounded with $T$-bound less than $1$. Then $T+S$ is closed if and only if $T$ is closed.
\end{theorem}

\begin{proof}
\cite[Theorem 5.5]{Weidmann}
\end{proof}

\begin{lemma}
If $T\in\mathcal{C}(E,F)$, then $\ker(T)\subset E$ is closed.
\end{lemma}

\begin{proof}
This follows from 

\[\ker(T)\times\{0\}=\gra(T)\cap(E\times\{0\}),\]
as $\gra(T)$ and $E\times\{0\}$ are closed.
\end{proof}
\noindent
The following important result is usually called the \textit{closed graph theorem}.

\begin{theorem}\label{closedgraph}
If $T\in\mathcal{C}(E,F)$ and $\mathcal{D}(T)=E$, then $T\in\mathcal{L}(E,F)$.
\end{theorem}

\begin{proof}
By assumption, $\gra(T)\subset E\times F$ is closed and hence a Banach space in its own right. We define an operator
\[S:E\rightarrow\gra(T),\quad Su=(u,Tu)\]
and note that $S$ is bijective and its inverse is the restriction $P_1\mid_{\gra(T)}:\gra(T)\rightarrow E$, where $P_1$ denotes the projection onto the first component in $E\times F$. As $P_1\mid_{\gra(T)}$ is bounded and surjective, it is an open map by the open mapping theorem (Theorem \ref{openmapping}). Accordingly, 

\[(P_1\mid_{\gra(T)})^{-1}=S\]
is bounded. To show the assertion, we just have to observe that $T=P_2\circ S$, where $P_2$ denotes the projection on the second component in $E\times F$. 
\end{proof}
\noindent 
Note that Theorem \ref{closedgraph} implies that $\mathcal{D}(T)\neq E$ if $T\in\mathcal{C}(E,F)$ is unbounded. Moreover, if $T\in\mathcal{C}(E,F)$ is invertible, then $T^{-1}\in\mathcal{C}(F,E)$ by Lemma \ref{linearprop} and $\mathcal{D}(T^{-1})=F$. Accordingly, $T^{-1}\in\mathcal{L}(F,E)$, and so a closed operator is invertible if and only if it has a bounded inverse.

\begin{lemma}\label{boundedcomposition}
If $B\in\mathcal{L}(E,F)$ and $T\in\mathcal{C}(F,G)$ are such that $\im(B)\subset\mathcal{D}(T)$, then $TB\in\mathcal{L}(E,G)$.
\end{lemma}  
 
\begin{proof}
By Theorem \ref{closedgraph} we only need to show that $TB$ is closed since $\mathcal{D}(TB)=E$ by assumption. Let $\{u_n\}_{n\in\mathbb{N}}$ be a sequence in $E$ such that $u_n\rightarrow u\in E$ and assume that $(TB)u_n\rightarrow v$, $n\rightarrow\infty$. As $B$ is bounded, we obtain $Bu_n\rightarrow Bu$. Moreover, since $T$ is closed, we get that $Bu\in\mathcal{D}(T)$ and $(TB)(u_n)\rightarrow T(Bu)$. Hence $v=TBu$ and so $TB$ is closed.
\end{proof}
\noindent 
Now let $T:\mathcal{D}(T)\subset E\rightarrow F$ be a linear operator. The \textit{graph norm} on $\mathcal{D}(T)$ is defined by
\[\|u\|_T=(\|u\|^2+\|Tu\|^2)^\frac{1}{2},\quad u\in\mathcal{D}(T).\]
Note that $T:\mathcal{D}(T)\rightarrow F$ is bounded when considered as an operator on the normed linear space $\mathcal{D}(T)$ with respect to the graph norm. The normed linear space $\mathcal{D}(T)$ with respect to $\|\cdot\|_T$ is isometric to the linear subspace $\gra(T)$ of $E\times F$. Hence we immediately obtain the following result.

\begin{lemma}
$T\in\mathcal{C}(E,F)$ if and only if $\mathcal{D}(T)$ is a Banach space with respect to the graph norm.
\end{lemma} 
\noindent
The following result shows that the Banach space structures are essentially determined by the underlying domains of the operators.

\begin{lemma}
If $T_1,T_2\in\mathcal{C}(E,F)$ have the same domain $D=\mathcal{D}(T_1)=\mathcal{D}(T_2)$, then the graph norms of $T_1$ and $T_2$ on $D$ are equivalent. 
\end{lemma}

\begin{proof}
We denote by $D_{T_1}$ and $D_{T_2}$ the space $D$ with the graph norm of $T_1$ and $T_2$, respectively. Let us consider $T_1:D_{T_2}\rightarrow F$ and let $\{(u_n,T_1u_n)\}_{n\in\mathbb{N}}\subset\gra(T_1:D_{T_2}\rightarrow F)$ be any sequence which converges in $D_{T_2}\times F$ to an element $(u,v)\in D_{T_2}\times F$. Since $u_n\rightarrow u$ in $D_{T_2}$ implies that $u_n$ converges also to $u$ with respect to the norm of $E$, we get from the closedness of $T_1$ that $(u_n,T_1u_n)\rightarrow (u,T_1u)$ in $E\times F$. Hence $(u_n,T_1u_n)\rightarrow (u,T_1u)$ in $D_{T_2}\times F$ which implies that $T_1:D_{T_2}\rightarrow F$ is closed. As $D_{T_2}$ is complete, we obtain from Theorem \ref{closedgraph} that $T_1:D_{T_2}\rightarrow F$ is bounded. Accordingly, there exists a constant $c>0$ such that
\[\|T_1u\|_F\leq c(\|u\|_E+\|T_2u\|_F)\]
and thus

\[\|u\|_E+\|T_1u\|_F\leq (c+1)(\|u\|_E+\|T_2u\|_F).\]
By using elementary inequalities, we finally obtain

\[(\|u\|^2_E+\|T_1u\|^2_F)^{\frac{1}{2}}\leq\sqrt{2}(c+1)(\|u\|^2_E+\|T_2u\|^2_F)^{\frac{1}{2}}.\]
The assertion follows by swapping $T_1$ and $T_2$.
\end{proof}

\noindent
We now consider special classes of closed Operators.

\begin{defi}
An operator $T\in\mathcal{C}(E,F)$ is called \textit{Fredholm} if $\ker(T)$ is of finite dimension and $\im(T)$ of finite codimension. The \textit{Fredholm index} of a Fredholm operator is defined as

\begin{align*}
\ind(T)=\dim\ker(T)-\codim\im(T).
\end{align*}
\end{defi}
\noindent
Note that $T$ is assumed to be a closed operator. The following result is often required as an additional assumption in the definition of a Fredholm operator.

\begin{lemma}
If $T\in\mathcal{C}(E,F)$ is Fredholm, then $\im(T)\subset F$ is closed.
\end{lemma}

\begin{proof}
If we consider $T$ on its domain $\mathcal{D}(T)$ with respect to the graph norm, then $T$ is bounded. As a change of the norm in the domain does not affect $\im(T)$, the assertion follows from Lemma \ref{codimclosed}.
\end{proof}

\begin{example}
Let us consider once again the operator $T:C^1[0,1]\subset C[0,1]\rightarrow C[0,1]$ given by $Tu=u'$. We have already seen in Example \ref{example-closed} that $T$ is closed. The kernel of $T$ is the one-dimensional subspace of $C[0,1]$ consisting of the constant functions. Given $v\in C[0,1]$, then $u(t)=\int^s_0{v(s)\,ds}$, $t\in[0,1]$, defines an element in $C^1[0,1]$ such that $Tu=v$. Hence $T$ is surjective and we conclude that $T$ is a Fredholm operator of index $1$.
\end{example}

\noindent
Finally, we mention without proof the following two results on the stability of Fredholmness.

\begin{theorem}\label{stabFred}
Let $T\in\mathcal{C}(E,F)$ be Fredholm. Then there exists $\gamma>0$ such that for every $S:\mathcal{D}(S)\subset E\rightarrow F$ such that $\mathcal{D}(T)\subset\mathcal{D}(S)$ and
\[\|Su\|\leq\gamma(\|u\|+\|Tu\|),\quad u\in\mathcal{D}(T),\]
the operator $T+S$ is Fredholm and

\begin{enumerate}
	\item[(i)] $\dim\ker(T+S)\leq\dim\ker(T)$,
	\item[(ii)] $\codim\im(T+S)\leq\codim\im(T)$,
	\item[(iii)] $\ind(T+S)=\ind(T)$. 
\end{enumerate} 
\end{theorem}

\begin{proof}
\cite[Theorem XVII.4.2]{GohbergClasses}
\end{proof}

\begin{theorem}\label{stabFredcomp}
Let $T\in\mathcal{C}(E,F)$ be Fredholm and $S:\mathcal{D}(S)\subset E\rightarrow F$ such that $\mathcal{D}(T)\subset\mathcal{D}(S)$ and $S\mid_{\mathcal{D}(T)}:\mathcal{D}(T)\rightarrow F$ is compact with respect to the graph norm of $T$ on $\mathcal{D}(T)$. Then $T+S$ is Fredholm and

\[\ind(T+S)=\ind(T).\]
\end{theorem}

\begin{proof}
\cite[Theorem XVII.4.3]{GohbergClasses}
\end{proof}


\section{Spectral Theory}
Let $T:\mathcal{D}(T)\subset E\rightarrow E$ be a linear operator. We call $\lambda\in\mathbb{C}$ an \textit{eigenvalue} of $T$ if there exists $u\in\mathcal{D}(T)$, $u\neq 0$, such that $Tu=\lambda u$. If $\lambda$ is not an eigenvalue, then $\lambda-T$ is injective and hence the \textit{resolvent operator}

\[R(\lambda,T)=(\lambda-T)^{-1},\quad\mathcal{D}(R(\lambda,T))=\im(T)\subset E,\]
is well defined. We define the \textit{resolvent set} of $T$ by

\[\rho(T)=\{\lambda\in\mathbb{C}:\lambda-T\,\,\,\text{bijective},\,\, R(\lambda,T)\in\mathcal{L}(E)\}.\]
A first important observation reads as follows:

\begin{lemma}
If $T\notin\mathcal{C}(E)$, then $\rho(T)=\emptyset$.
\end{lemma}

\begin{proof}
If $\lambda\in\rho(T)\neq\emptyset$, then $R(\lambda,T)=(\lambda-T)^{-1}$ is bounded and hence closed by Lemma \ref{boundedclosed}. Then $\lambda-T$ is closed by Lemma \ref{linearprop} and so $T$ is closed as well. 
\end{proof}

\noindent
Accordingly, we assume in what follows that $T\in\mathcal{C}(E)$. Note that in this case we obtain from Theorem \ref{closedgraph} that

\begin{align*}
\rho(T)=\{\lambda\in\mathbb{C}:\lambda-T\,\,\,\text{is bijective}\}.
\end{align*}
We define the \textit{spectrum} $\sigma(T)$ of $T$ to be $\mathbb{C}\setminus\rho(T)$ and the \textit{point spectrum} $\sigma_p(T)\subset\sigma(T)$ as the set of all eigenvalues of $T$. 

\begin{lemma}
The resolvent set $\rho(T)\subset\mathbb{C}$ is open and, accordingly, $\sigma(T)\subset\mathbb{C}$ is closed. Moreover, if $\lambda_0\in\rho(T)\neq \emptyset$ and $|\lambda-\lambda_0|<\|R(\lambda_0,T)\|^{-1}$, then $\lambda\in\rho(T)$ and

\[R(\lambda,T)=\sum^{\infty}_{k=0}{(\lambda-\lambda_0)^kR(\lambda_0,T)^{k+1}},\]
where the series converges in the norm of $\mathcal{L}(E)$. In particular, $R(\cdot,T)$ is analytic on $\rho(T)$.
\end{lemma}

\begin{proof}
We assume that $\rho(T)\neq\emptyset$ and take some $\lambda_0\in\rho(T)$. If $|\lambda-\lambda_0|<\|R(\lambda_0,T)\|^{-1}$ then the series

\begin{align*}
S(\lambda):=\sum^{\infty}_{k=0}{(\lambda-\lambda_0)^kR(\lambda_0,T)^{k+1}}
\end{align*}
converges absolutely and we claim that $S(\lambda)=R(\lambda,T)=(T-\lambda)^{-1}$. If $u\in \mathcal{D}(T)$, then

\begin{align}\label{equ:resolvent}
R(\lambda_0,T)(T-\lambda)u&=(T-\lambda_0)^{-1}(T-\lambda_0-(\lambda-\lambda_0))u=u-(\lambda-\lambda_0)R(\lambda_0,T)u
\end{align} 
and we obtain

\begin{align*}
S(\lambda)(T-\lambda)u&=\sum^\infty_{k=0}{(\lambda-\lambda_0)^kR(\lambda_0,T)^k(u-(\lambda-\lambda_0)R(\lambda_0,T)u)}\\
&=\sum^\infty_{k=0}{(\lambda-\lambda_0)^kR(\lambda_0,T)^ku}-\sum^\infty_{k=0}{(\lambda-\lambda_0)^{k+1}R(\lambda_0,T)^{k+1}u}=u.
\end{align*}
Consequently, $S(\lambda)(T-\lambda)=I_{\mathcal{D}(T)}$. Now let $v\in E$ be given. We define
\[u_n:=\sum^n_{k=0}{(\lambda-\lambda_0)^kR(\lambda_0)^{k+1}v},\quad n\in\mathbb{N},\]
which is a sequence in $\mathcal{D}(T)$ converging to $u:=S(\lambda)v\in E$. By using a computation as in \eqref{equ:resolvent}, we obtain

\begin{align*}
(T-\lambda)u_n&=\sum^n_{k=0}{(\lambda-\lambda_0)^k(T-\lambda)R(\lambda_0,T)^{k+1}v}\\
&=\sum^n_{k=0}{(\lambda-\lambda_0)^k(I_H-(\lambda-\lambda_0)R(\lambda_0,T))R(\lambda_0,T)^{k}v}\\
&=\sum^n_{k=0}{(\lambda-\lambda_0)^kR(\lambda_0,T)^{k}v}-\sum^n_{k=0}{(\lambda-\lambda_0)^{k+1}R(\lambda_0,T)^{k+1}v}\\&=v-(\lambda-\lambda_0)^{n+1}R(\lambda_0,T)^{n+1}v
\end{align*} 
and so we see from our assumption on $\lambda$ that $(T-\lambda)u_n\rightarrow v\in E$, $n\rightarrow\infty$. As $T-\lambda$ is closed, we conclude that $u\in\mathcal{D}(T)$ and $(T-\lambda)S(\lambda)v=(T-\lambda)u=v$. Consequently, $(T-\lambda)S(\lambda)=I_E$ and we finally obtain that $S(\lambda)=(T-\lambda)^{-1}=R(\lambda,T)$.
\end{proof}

\noindent
Let us point out that in contrast to operators in $\mathcal{L}(E)$, the spectrum of elements in $\mathcal{C}(E)$ can be unbounded and even empty:

\begin{example}
We consider again the closed operators

\[T:\mathcal{D}(T)\subset C[0,1]\rightarrow C[0,1],\quad T_0:\mathcal{D}(T_0)\subset C[0,1]\rightarrow C[0,1],\]
where 

\[\mathcal{D}(T)=C^1[0,1],\qquad \mathcal{D}(T_0)=\{u\in C^1[0,1]:\, u(0)=0\}\]
and both operators map elements $u$ to their first derivative.\\
For the operator $T$, we see that $e^{\lambda\,t}\in\ker(\lambda-T)$ for all $\lambda\in\mathbb{C}$ and hence $\sigma(T)=\mathbb{C}$. For $T_0$ it is readily seen that $\ker(\lambda-T_0)=\{0\}$ for all $\lambda\in\mathbb{C}$. Moreover, if $g\in C[0,1]$, then $u(t)=-e^{\lambda t}\int^t_0{e^{-\lambda s}\,g(s)\,ds}$, $t\in[0,1]$, is a solution of the equation $(\lambda-T_0)u=g$ and so

\[\|(\lambda-T_0)^{-1}g\|=\|u\|\leq e^{|\lambda|}\|g\|.\]
Hence $\sigma(T_0)=\emptyset$.
\end{example}

\noindent
There are various definitions of subsets of $\sigma(T)$ other than $\sigma_p(T)$. Here we just want to introduce the \textit{essential spectrum} $\sigma_{ess}(T)$ which consists of all $\lambda\in\mathbb{C}$ such that $\lambda-T$ is not a Fredholm operator. In contrast to the case of bounded operators, it is possible that $\sigma_{ess}(T)=\emptyset$. Note that in general we neither have $\sigma(T)=\sigma_p(T)\cup\sigma_{ess}(T)$ nor $\sigma_p(T)\cap\sigma_{ess}(T)=\emptyset$.\\
As next step we introduce \textit{Riesz projections}. Let $T\in\mathcal{C}(E)$ and assume that we have a disjoint union $\sigma(T)=\sigma\cup\tau$, where $\sigma$ is contained in a bounded Cauchy domain\footnote{The definition of Cauchy domain can be found in \cite[Sec. I.1]{GohbergClasses}. For example, any open and connected subset of $\mathbb{C}$ whose boundary is a closed rectifiable Jordan curve is a Cauchy domain.} $\Delta$ such that $\tau\cap\overline{\Delta}=\emptyset$. Let $\Gamma$ denote the oriented boundary of $\Delta$. For $N\in\mathbb{N}$ sufficiently large we approximate $\Gamma$ by a union of straight-line segments $\Delta_i$, $i=1,\ldots N$, such that $\bigcup^N_{i=1}\Delta_i$ is a closed polygonal contour containing $\sigma$. We choose $\lambda_i\in\Delta_i$, $i=1,\ldots,N$ and set

\begin{align*}
P^N_\sigma=\frac{1}{2\pi i}\sum^N_{i=1}{R(\lambda_i,T)|\Delta_i|},
\end{align*}
where $|\Delta_i|$ denotes the length of the segment $\Delta_i$. In this way we can construct a sequence $\{P^N_\sigma\}\subset\mathcal{L}(H)$ such that the lengths of the straight-line segments $\Lambda_i$ converge to $0$. One can show that the limit $\lim_{N\rightarrow\infty}P^N_\sigma$ exists in $\mathcal{L}(E)$ and does not depend on the choices made in the construction of $\{P^N_\sigma\}$. We denote the obtained operator by

\begin{align*}
\frac{1}{2\pi i}\int_{\Gamma}{(\lambda-T)^{-1}\,d\lambda}.
\end{align*}
Given an operator $T$ acting on a linear space $E$, a subspace $M\subset E$ is called $T$-invariant if $T(M\cap\mathcal{D}(T))\subset M$. In this case $T\mid_M$ denotes the operator $T$ with domain $M\cap\mathcal{D}(T)$ and range in $M$.

\begin{theorem}\label{spectralprojections}
Let $T\in\mathcal{C}(E)$ with spectrum $\sigma(T)=\sigma\cup\tau$, where $\sigma$ is contained in a bounded Cauchy domain $\bigtriangleup$ such that $\overline{\bigtriangleup}\cap\tau=\emptyset$. Let $\Gamma$ be the oriented boundary of $\bigtriangleup$. Then

\begin{enumerate}
	\item[(i)] $P_\sigma:=\frac{1}{2\pi i}\int_{\Gamma}{(\lambda-T)^{-1}\,d\lambda}$ does not depend on the choice of $\bigtriangleup$,
	\item[(ii)] $P_\sigma$ is a projection, i.e., $P^2_\sigma=P_\sigma$,
	\item[(iii)] the subspaces $M=\im(P_\sigma)$ and $N=\ker(P_\sigma)$ are $T$-invariant,
	\item[(iv)] the subspace $M$ is contained in $\mathcal{D}(T)$ and $T\mid_M$ is bounded,
	\item[(v)] $\sigma(T\mid_M)=\sigma$ and $\sigma(T\mid_N)=\tau$. 
\end{enumerate}
\end{theorem}

\begin{proof}
\cite[Theorem XV.2.1]{GohbergClasses}
\end{proof}

\noindent
$P_\sigma$ is called the \textit{Riesz projection} of the operator $T$ with respect to $\sigma$.\\
We call a point $\lambda\in\sigma(T)$ an \textit{eigenvalue of finite type} if $\lambda$ is isolated in $\sigma(T)$ and the associated projection $P_{\{\lambda\}}$ has finite rank. Since by Theorem \ref{spectralprojections}, $\{\lambda\}=\sigma(T\mid_{\im(P_{\{\lambda\}})})$ and $\im(P_{\{\lambda\}})$ is finite dimensional, it follows that $\lambda\in\sigma_p(T)$. Moreover, note that $\im(P_{\{\lambda\}})\supset\ker(\lambda-T)$ but equality does not hold in general. However, we will see in the next chapter that $\im(P_{\{\lambda\}})=\ker(\lambda-T)$ if $T$ is selfadjoint.\\
A quite often appearing situation is considered in the following theorem.

\begin{theorem}\label{compres}
Let $T\in\mathcal{C}(E)$ be such that $R(\lambda_0,T)=(\lambda_0-T)^{-1}$ is compact for some $\lambda_0\in\rho(T)$. Then $R(\lambda,T)$ is compact for any $\lambda\in\rho(T)$, $\sigma(T)$ does not have a limit point in $\mathbb{C}$ and every point in $\sigma(T)$ is an eigenvalue of finite type. Moreover, for any $\lambda\in\mathbb{C}$ we have

\begin{align*}
\dim\ker(\lambda-T)=\codim\im(\lambda-T)<\infty
\end{align*}
so that each operator $\lambda-T$ is Fredholm of index $0$.
\end{theorem}

\begin{proof}
\cite[Theorem XV.2.3]{GohbergClasses}
\end{proof}

\noindent
We call $T\in\mathcal{C}(E)$ an \textit{operator with compact resolvent} if there exists $\lambda_0\in\mathbb{C}$ such that $R(\lambda_0,T)$ is compact. Note that $\sigma_{ess}(T)=\emptyset$ in this case.

\begin{lemma}\label{Fredholmperturbation}
If $T\in\mathcal{C}(E)$ has a compact resolvent and $B\in\mathcal{L}(E)$ is bounded, then $T+B$ is a Fredholm operator of index $0$.
\end{lemma}

\begin{proof}
We fix $\lambda\in\rho(T)$ and obtain a compact operator $(\lambda-T)^{-1}:E\rightarrow E$. Then $B(\lambda-T)^{-1}:E\rightarrow E$ is compact as well since the composition of a bounded and a compact operator is still compact. Consequently,

\begin{align*}
B=(B(\lambda-T)^{-1})(\lambda-T):\mathcal{D}(T)\rightarrow E
\end{align*}
is compact with respect to the graph norm of $T$ on $\mathcal{D}(T)$. Now the assertion follows from Theorem \ref{stabFredcomp} and Theorem \ref{compres}.
\end{proof}


\chapter{Selfadjoint Operators}

\section{Definitions and Basic Properties}
Let $H$ be a complex non-trivial Hilbert space. The following result is known as the Riesz representation theorem.

\begin{theorem}
Let $f:H\rightarrow\mathbb{C}$ be a continuous linear functional. Then there exists a unique $v\in H$ such that
\[f(u)=\langle u,v\rangle,\quad u\in H.\]
\end{theorem} 
\noindent
Let $T:\mathcal{D}(T)\subset H\rightarrow H$ be a densely defined operator acting on $H$. We define

\begin{align*}
\mathcal{D}(T^\ast)=\{v\in H: u\mapsto\langle Tu,v\rangle_H\,\text{is bounded on}\,\mathcal{D}(T)\},
\end{align*}
and note that, as $\mathcal{D}(T)$ is assumed to be dense in $H$, each functional $u\mapsto\langle Tu,v\rangle_H$, $v\in\mathcal{D}(T^\ast)$, has a continuous extension on all of $H$. Hence, by the Riesz Representation Theorem, we can associate to any $v\in\mathcal{D}(T^\ast)$ an element $T^\ast v\in H$ such that

\begin{align*}
\langle Tu,v\rangle_H=\langle u,T^\ast v\rangle_H,\quad u\in\mathcal{D}(T),\, v\in\mathcal{D}(T^\ast).
\end{align*}
The resulting operator $T^\ast$ on $H$ with domain $\mathcal{D}(T^\ast)$ is linear and it is called the \textit{adjoint} of $T$. We make at first a rather simple observation, which is however often used when dealing with adjoints of operators.

\begin{lemma}
Assume that $T:\mathcal{D}(T)\subset H\rightarrow H$ is densely defined and $S:\mathcal{D}(S)\subset H\rightarrow H$ is any linear operator. If

\begin{align}\label{obs}
\langle Tu,v\rangle=\langle u,Sv\rangle,\quad u\in\mathcal{D}(T),\,v\in\mathcal{D}(S),
\end{align}
then $S\subset T^\ast$.
\end{lemma}

\begin{proof}
We obtain from \eqref{obs} and the definition of the adjoint that $\mathcal{D}(S)\subset\mathcal{D}(T^\ast)$. Hence we conclude from \eqref{obs} that $\langle u,T^\ast v\rangle=\langle u,Sv\rangle$ for all $u\in\mathcal{D}(T)$ and $v\in\mathcal{D}(S)$. This implies that $Sv=T^\ast v$, as $\mathcal{D}(T)$ is dense by assumption.
\end{proof}

\begin{lemma}\label{APPFA-lemma-adjointclosed}
If $T$ is densely defined, then $T^\ast\in\mathcal{C}(H)$.
\end{lemma}

\begin{proof}
Let us assume that $\{u_n\}_{n\in\mathbb{N}}\subset\mathcal{D}(T^\ast)$ is such that $u_n\rightarrow u$ and $T^\ast u_n\rightarrow v$, $n\rightarrow\infty$, for some $u,v\in H$. We obtain for all $w\in\mathcal{D}(T)$ that $\langle w, T^\ast u_n\rangle=\langle Tw,u_n\rangle\rightarrow\langle Tw,u\rangle$ and $\langle w,T^\ast u_n\rangle\rightarrow \langle w,v\rangle$. Hence $\langle w,v\rangle=\langle Tw,u\rangle$ for all $w\in\mathcal{D}(T)$ which shows that $u\in\mathcal{D}(T^\ast)$ by the definition of $T^\ast$. Moreover, we obtain $\langle w,v\rangle=\langle w,T^\ast u\rangle$ for all $w\in\mathcal{D}(T)$ and conclude that $v=T^\ast u$. This shows that $T^\ast$ is closed.
\end{proof}
\noindent
Let us point out that it is possible that $\mathcal{D}(T^\ast)=\{0\}$ for a densely defined operator $T$ (cf.\cite[p.291]{GohbergClasses}). However, since we are solely interested in closed operators, such singular phenomena cannot occur:

\begin{lemma}\label{APPFA-lemma-adjoint}
If $T\in\mathcal{C}(H)$ is densely defined, then $T^\ast$ is densely defined as well and $T^{\ast\ast}=T$.
\end{lemma} 

\begin{proof}
\cite[Theorem 5.3]{Weidmann}
\end{proof}

\noindent
A densely defined operator $T$ acting on a Hilbert space $H$ is called \textit{symmetric} if $T\subset T^\ast$ and \textit{selfadjoint} if $T=T^\ast$. Note that a selfadjoint operator is closed by Lemma \ref{APPFA-lemma-adjointclosed}. 

\begin{lemma}
Let $T$ be a densely defined operator acting on $H$.

\begin{enumerate}
	\item[(i)] $T$ is symmetric if and only if $T\subset T^{\ast\ast}\subset T^\ast=T^{\ast\ast\ast}$. In this case $T^{\ast\ast}$ is symmetric as well.
	\item[(ii)] $T\in\mathcal{C}(H)$ is symmetric if and only if $T=T^{\ast\ast}\subset T^\ast$.
	\item[(iii)] $T$ is selfadjoint if and only if $T=T^{\ast\ast}=T^\ast$. 
\end{enumerate}
\end{lemma}

\begin{proof}
\cite[VII.2.5]{Werner}
\end{proof}

\noindent
The next result is the \textit{Hellinger-Toeplitz} theorem:

\begin{theorem}
If $T$ is symmetric and $\mathcal{D}(T)=H$, then $T$ is selfadjoint and $T\in\mathcal{L}(H)$.
\end{theorem}

\begin{proof}
Since $T\subset T^\ast$ and $\mathcal{D}(T)=H$, it is clear that $T$ is selfadjoint. In particular, $T$ is closed by Lemma \ref{APPFA-lemma-adjointclosed} and hence bounded by the closed graph theorem (Theorem \ref{closedgraph}).
\end{proof}
\noindent
The following lemma holds analogously to the bounded case:

\begin{lemma}\label{selfadjointperp}
Let $T:\mathcal{D}(T)\subset H\rightarrow H$ be densely defined. Then

\begin{enumerate}
	\item[(i)] $(\im T)^\perp=\ker(T^\ast)$,
	\item[(ii)] $\overline{\im T}=(\ker T^\ast)^\perp$. 
\end{enumerate}
\end{lemma}

\begin{proof}
Since $(U^\perp)^\perp=\overline{U}$ for all $U\subset H$, the second assertion follows from the first one. In order to show (i) let us assume at first that $u\in(\im T)^\perp$. Then $\langle Tv,u\rangle=0$ for all $v\in\mathcal{D}(T)$ and we conclude that $u\in\mathcal{D}(T^\ast)$ and $T^\ast u=0$. Consequently, $u\in\ker(T^\ast)$. If, conversely, $v\in\ker(T^\ast)$, then $\langle u,T^\ast v\rangle=0$ for all $u\in H$. Hence $\langle Tu,v\rangle=0$ for all $u\in\mathcal{D}(T)$ and so $v\in(\im T)^\perp$.
\end{proof}

\begin{lemma}\label{adjointinvertible}
If $T\in\mathcal{C}(H)$ is densely defined and invertible, then $T^\ast$ is invertible as well and $(T^\ast)^{-1}=(T^{-1})^\ast$.
\end{lemma}

\begin{proof}
As $T^{-1}:H\rightarrow H$, the domain of the adjoint $(T^{-1})^\ast$ is $H$ and we obtain for all $v\in H$ and $u\in\mathcal{D}(T)$

\[\langle Tu,(T^{-1})^\ast v\rangle=\langle T^{-1}Tu,v\rangle=\langle u,v\rangle.\]
Hence

\begin{align}\label{equ:adjointinvertible}
(T^{-1})^\ast v\in\mathcal{D}(T^\ast),\quad T^\ast(T^{-1})^\ast v=v,\quad v\in H, 
\end{align}
and we note that $T^\ast$ is surjective. As $T$ is invertible, $T^\ast$ is injective as well by Lemma \ref{selfadjointperp}, and so $T^\ast$ has a bounded inverse by the closed graph theorem. Finally, from \eqref{equ:adjointinvertible} we deduce that $(T^\ast)^{-1}=(T^{-1})^\ast$. 
\end{proof}

\begin{lemma}\label{adjointscommute}
Let $T,S\in\mathcal{C}(H)$ be densely defined. If $ST$ is densely defined, then
\[T^\ast S^\ast\subset (ST)^\ast.\]
Moreover, if $S\in\mathcal{L}(H)$ then
\[(ST)^\ast=T^\ast S^\ast.\]
\end{lemma}

\begin{proof}
If $u\in\mathcal{D}(ST)$ and $v\in\mathcal{D}(T^\ast S^\ast)$, then

\[\langle Tu,S^\ast v\rangle=\langle u,T^\ast S^\ast v\rangle\]
and

\[\langle STu,v\rangle=\langle Tu,S^\ast v\rangle.\]
Consequently, $\langle STu,v\rangle=\langle u,T^\ast S^\ast v\rangle$ and the first assertion is shown.\\
Now assume that $S\in\mathcal{L}(H)$ and $v\in\mathcal{D}((ST)^\ast)$. As $S^\ast\in\mathcal{L}(H)$, we have $\mathcal{D}(S^\ast)=H$ and so we obtain for all $u\in\mathcal{D}(ST)$

\[\langle Tu,S^\ast v\rangle=\langle STu,v\rangle=\langle u,(ST)^\ast v\rangle.\]
We conclude that $S^\ast v\in\mathcal{D}(T^\ast)$ and so $v\in\mathcal{D}(T^\ast S^\ast)$. Consequently, $(ST)^\ast= T^\ast S^\ast$ by the first assertion of the lemma.
\end{proof}
\noindent
The following example can be found, e.g., in \cite{Rudin}.

\begin{example}\label{example:T}
We consider $H=L^2[0,1]$ and define

\begin{align*}
\mathcal{D}(T_1)=H^1[0,1]=\{u:[0,1]\rightarrow\mathbb{C}:\,u\,\text{ absolutely continuous, }\, u'\in L^2[0,1]\},
\end{align*}
as well as

\begin{align*}
\mathcal{D}(T_2)&=\{u\in\mathcal{D}(T_1):\,u(0)=u(1)\}\\
\mathcal{D}(T_3)&=\{u\in\mathcal{D}(T_1):\, u(0)=u(1)=0\}.
\end{align*}
If we set $T_ku=i\,u'$ for $k=1,2,3$, then

\begin{align}\label{selfadj:equ}
T^\ast_1=T_3,\quad T^\ast_2=T_2,\quad T^\ast_3=T_1.
\end{align}
As $T_3\subset T_2\subset T_1$ we note in particular that $T_3$ is symmetric and $T_2$ is selfadjoint.\\
\noindent
In order to show \eqref{selfadj:equ} we compute for $u\in\mathcal{D}(T_k)$, $v\in\mathcal{D}(T_m)$, $m+k=4$,

\begin{align*}
\langle T_ku,v\rangle&=\int^1_0{(iu')\overline{v}\,dt}=\underbrace{iu(1)\overline{v(1)}-iu(0)\overline{v(0)}}_{=0}-i\int^1_0{u\overline{v}'\,dt}=\langle u,T_mv\rangle
\end{align*}
and see that

\[T_1\subset T^\ast_3,\quad T_2\subset T^\ast_2,\quad T_3\subset T^\ast_1.\]
Now we assume that $v\in\mathcal{D}(T^\ast_k)$ and we set $w(t):=\int^t_0{T^\ast_kv\,ds}$, $t\in[0,1]$. We obtain for $u\in\mathcal{D}(T_k)$

\begin{align}\label{selfadj:exampleformulaI}
\int^1_0{iu'\overline{v}\,dt}=\langle T_ku,v\rangle=\langle u,T^\ast_kv\rangle=u(1)\overline{w(1)}-\underbrace{u(0)\overline{w(0)}}_{=0}-\int^1_0{u'\overline{w}\,dt}.
\end{align}
If now $k=1$ or $k=2$, then $Y:=\{u\in L^2[0,1]:\, u\equiv const.\}\subset\mathcal{D}(T_k)$ and we see from \eqref{selfadj:exampleformulaI} that

\begin{align}\label{selfadj:exampleformulaII}
w(1)=0,\quad k=1,2.
\end{align}
If $k=3$, we have $u(1)=0$ for all $u\in\mathcal{D}(T_3)$ and we conclude that in all cases

\begin{align}\label{selfadj:exampleformulaIII}
iv-w\in(\im T_k)^\perp,\quad k=1,2,3.
\end{align}
Now let us consider at first $T_1$. As $T_1$ is surjective, we obtain from \eqref{selfadj:exampleformulaIII} that $iv=w$. By \eqref{selfadj:exampleformulaII}, this implies that $v(0)=v(1)=0$ and hence $v\in\mathcal{D}(T_3)$ which shows that $T^\ast_1\subset T_3$.\\
For $k=2$ and $k=3$, we note at first that $\im(T_k)=Y^\perp$ and so we obtain from \eqref{selfadj:exampleformulaIII} that $iv-w$ is a constant function. By the definition of $w$, we conclude that $v$ is absolutely continuous and $v'=T^\ast_kv\in L^2[0,1]$. Hence $v\in\mathcal{D}(T_1)$, and for $k=3$ we see that $T^\ast_3\subset T_1$. For $k=2$ we have in addition $w(1)=0$ by \eqref{selfadj:exampleformulaII} which shows $v(0)=v(1)$. Hence $v\in\mathcal{D}(T_2)$ and $T^\ast_2\subset T_2$.  
\end{example}

\begin{lemma}\label{lemma:symsurj}
If $T:\mathcal{D}(T)\subset H\rightarrow H$ is symmetric and surjective, then $T$ is selfadjoint.
\end{lemma}

\begin{proof}
As $T$ is symmetric, we only need to show that $\mathcal{D}(T^\ast)\subset\mathcal{D}(T)$. Let $v\in\mathcal{D}(T^\ast)$. Since $T$ is surjective, there exists $u\in\mathcal{D}(T)$ such that $Tu=T^\ast v$. So for every $w\in\mathcal{D}(T)$

\[\langle Tw,v\rangle=\langle w,T^\ast v\rangle=\langle w,Tu\rangle=\langle Tw,u\rangle,\]
and consequently $v=u\in\mathcal{D}(T)$, where we use again the surjectivity of $T$. 
\end{proof}

\begin{example}\label{example:selfadjointFredholm}
We revisit the differential operator $T_2$ on $L^2[0,1]$ from Example \ref{example:T}, where we modify the domain slightly. We consider for $\lambda\in[-\pi,\pi]$ the operator $T_\lambda u=i\,u'$ on

\[\mathcal{D}(T_\lambda)=\{u\in H^1[0,1]:\, u(0)=e^{i\lambda} u(1)\}.\]
Note that $T_\lambda=T_2$ for $\lambda=0$, so that in this case $T_\lambda$ is selfadjoint by Example \ref{example:T}. If, however, $\lambda\neq 0$ and $v\in L^2[0,1]$, then a straightforward computation shows that

\[u(t)=-i\int^t_0{v(s)\,ds}-\frac{i e^{i\lambda}}{1-e^{i\lambda}}\int^1_0{v(s)\,ds},\quad t\in[0,1],\]
belongs to $\mathcal{D}(T_\lambda)$ and $T_\lambda u=v$. Hence $T_\lambda$ is surjective for every $\lambda\neq 0$. Moreover, if $u,v\in\mathcal{D}(T_\lambda)$

\begin{align*}
\langle T_\lambda u,v\rangle&=\int^1_0{(iu')\overline{v}\,dt}=\underbrace{ie^{i\lambda}u(0)\overline{e^{i\lambda}v(0)}-iu(0)\overline{v(0)}}_{=0}-i\int^1_0{u\overline{v}'\,dt}=\langle u,T_\lambda v\rangle,
\end{align*}
and so $T_\lambda$ is symmetric. Consequently, by Lemma \ref{lemma:symsurj}, $T_\lambda$ is selfadjoint for all $\lambda\in[-\pi,\pi]$. 
\end{example}

\noindent
Note that the sum of two selfadjoint operators is in general not selfadjoint. The \textit{Kato-Rellich theorem} is a classical result giving conditions on how much a selfadjoint operator can be perturbed without losing its selfadjointness. Here we just prove a much weaker assertion which, however, is often sufficient for applications and is a rather direct consequence of the definitions.

\begin{theorem}
Let $T:\mathcal{D}(T)\subset H\rightarrow H$ be selfadjoint and let $S:H\rightarrow H$ be symmetric. Then $T+S$ is selfadjoint on $\mathcal{D}(T)$.
\end{theorem}

\begin{proof}
Let $v\in\mathcal{D}(T+S)=\mathcal{D}(T)=\mathcal{D}(T^\ast)$. Then

\begin{align*}
\langle(T+S)u,v\rangle=\langle Tu,v\rangle+\langle Su,v\rangle=\langle u,Tv\rangle+\langle u,Sv\rangle=\langle u,(T+S)v\rangle
\end{align*}
and the functional $u\mapsto\langle(T+S)u,v\rangle$ is bounded on $\mathcal{D}(T+S)=\mathcal{D}(T)$. Hence $v\in\mathcal{D}((T+S)^\ast)$ and we have shown that $T+S\subset(T+S)^\ast$.\\
Now we assume that $v\in\mathcal{D}((T+S)^\ast)$. Then the functional 

\begin{align*}
u\mapsto\langle(T+S)u,v\rangle=\langle Tu,v\rangle+\langle Su,v\rangle,\quad u\in\mathcal{D}(T+S)=\mathcal{D}(T),
\end{align*}
is bounded. If we now assume that $u\mapsto\langle Tu,v\rangle$ is unbounded on $\mathcal{D}(T)$, then there exists a sequence $\{u_n\}_{n\in\mathbb{N}}\subset\mathcal{D}(T)$, $\|u_n\|=1$, $n\in\mathbb{N}$, such that $\langle Tu_n,v\rangle\rightarrow\infty$, $n\rightarrow\infty$. As $|\langle Su_n,v\rangle|\leq\|S\|\|v\|$, we would obtain that $\langle (T+S)u_n,v\rangle\rightarrow\infty$ which is a contradiction. Hence $u\mapsto\langle Tu,v\rangle$ is bounded on $\mathcal{D}(T)$ and so $v\in\mathcal{D}(T^\ast)=\mathcal{D}(T)=\mathcal{D}(T+S)$. Consequently, $(T+S)^\ast= T+S$.
\end{proof}


\section{Spectral Theory of Selfadjoint Operators}

\begin{lemma}\label{selfadjoint-closedrange}
Let $E$ be a Banach space. If $T\in\mathcal{C}(E)$ and there exists $\beta\geq 0$ such that

\begin{align}\label{equ:selfadjoint-closedrange}
\|Tu\|\geq\beta\|u\|,\quad u\in\mathcal{D}(T),
\end{align}
then $\im(T)$ is closed.
\end{lemma}

\begin{proof}
Let $\{Tu_n\}_{n\in\mathbb{N}}$ be a sequence in $\im(T)$ such that $Tu_n\rightarrow v\in E$. We obtain from \eqref{equ:selfadjoint-closedrange} that $\{u_n\}_{n\in\mathbb{N}}$ is a Cauchy sequence in $E$ and so $u_n\rightarrow u$ for some $u\in E$. As $T$ is closed, we conclude that $u\in\mathcal{D}(T)$ and $v=Tu\in\im(T)$. Consequently, $\im(T)$ is closed.  
\end{proof}

\begin{lemma}
If $T:\mathcal{D}(T)\subset H\rightarrow H$ is selfadjoint, then $\sigma(T)\subset\mathbb{R}$ and 

\begin{align}\label{equ:selfadjoint-resolvent}
\|(\lambda-T)^{-1}\|\leq|\beta|^{-1},\quad \lambda=\alpha+i\,\beta,\,\alpha,\beta\in\mathbb{R},\, \beta\neq 0.
\end{align}
\end{lemma}

\begin{proof}
Assume that $\lambda=\alpha+i\,\beta$, $\alpha,\beta\in\mathbb{R}$ and $\beta\neq 0$. We obtain for $u\in\mathcal{D}(T)$

\begin{align*}
\langle (\alpha+i\beta)u-Tu,(\alpha+i\beta)u-Tu\rangle&=\langle(\alpha u-Tu)+i\beta u,(\alpha u-Tu)+i\beta u\rangle\\
&=\|\alpha u-Tu\|^2+\beta^2\|u\|^2-i\beta\langle (\alpha-T)u,u\rangle+i\beta\langle u,(\alpha-T)u\rangle
\end{align*}
and conclude that

\begin{align}\label{equpro:selfadjoint-resolvent}
\|(\lambda-T)u\|^2\geq\beta^2\|u\|^2,\quad u\in\mathcal{D}(T),
\end{align}
where we use that $T$ is symmetric. Hence $\im(\lambda-T)$ is closed by Lemma \ref{selfadjoint-closedrange}. As $\ker(\lambda-T)=0$ by \eqref{equpro:selfadjoint-resolvent}, we obtain from Lemma \ref{selfadjointperp}

\begin{align*}
H=\overline{\im(\lambda-T)}\oplus\ker(\lambda-T)=\im(\lambda-T)
\end{align*} 
and so $\lambda\in\rho(T)$. Finally, \eqref{equ:selfadjoint-resolvent} follows from \eqref{equpro:selfadjoint-resolvent}.
\end{proof}
\noindent
The first part of the previous lemma can be improved as follows:

\begin{lemma}
Let $T\in\mathcal{C}(H)$ be densely defined and symmetric. Then precisely one of the following assertions hold:

\begin{enumerate}
	\item[(i)] $\sigma(T)=\mathbb{C}$;
	\item[(ii)] $\sigma(T)=\{\alpha+i\beta\in\mathbb{C}:\beta\geq0\}$;
	\item[(iii)] $\sigma(T)=\{\alpha+i\beta\in\mathbb{C}:\beta\leq 0\}$;
	\item[(iv)] $\sigma(T)\subset\mathbb{R}$. 
\end{enumerate}
Moreover, $\sigma(T)\subset\mathbb{R}$ if and only if $T$ is selfadjoint.
\end{lemma}

\begin{proof}
\cite[Sec. V.3.4]{Kato}
\end{proof}
\noindent
Note that in particular a closed symmetric operator is selfadjoint if $\rho(T)\cap\mathbb{R}\neq\emptyset$.

\begin{lemma}\label{spectrumselfadjoint}
If $T\in\mathcal{C}(H)$ is densely defined and selfadjoint, then

\begin{align*}
\sigma(T)=\sigma_p(T)\cup\sigma_{ess}(T).
\end{align*}

\end{lemma}

\begin{proof}
Since $\sigma(T)\subset\mathbb{R}$, any $\lambda-T$, $\lambda\in\sigma(T)$, is selfadjoint as well. Hence we can assume without loss of generality that $0\in\sigma(T)$ and it suffices to consider the case $\lambda=0$.\\ 
As $\ker(T)=(\im T)^\perp$ by Lemma \ref{selfadjointperp}, we see that either $T$ is not injective or its image is dense in $H$. If $T$ is not injective, then $0\in\sigma_p(T)$. If on the other hand $T$ is injective, then $T$ has a dense image but $\im(T)\neq H$ as otherwise $0\in\rho(T)$. Hence $\im(T)$ is not closed and so $T$ is not a Fredholm operator which implies that $0\in\sigma_{ess}(T)$. 
\end{proof}
\noindent
Finally, we consider selfadjoint Fredholm operators.

\begin{lemma}\label{spectralFredSelf}
Let $T\in\mathcal{C}(H)$ be selfadjoint and Fredholm. Then $0$ is either in the resolvent set $\rho(T)$ or it is an isolated eigenvalue of finite multiplicity.
\end{lemma}

\begin{proof}
As $T$ is Fredholm, we get from Lemma \ref{stabFred} that there is $\varepsilon>0$ such that $\lambda-T$ is also Fredholm for all $\lambda\in(-\varepsilon,\varepsilon)$. Consequently, $(-\varepsilon,\varepsilon)\cap\sigma_{ess}(T)=\emptyset$ and by Lemma \ref{spectrumselfadjoint} we only need to show that $0$ is isolated in $\sigma(T)$ if it is not in the resolvent set.\\
We set $X:=(\ker T)^\perp=\im(T)$ which is a Hilbert space as it is a closed subspace of $H$. We claim that 

\begin{align*}
T':=T\mid_{X}:\mathcal{D}(T')=X\cap\mathcal{D}(T)\subset X\rightarrow X
\end{align*}
is closed. Indeed, if $\{u_n\}_{n\in\mathbb{N}}\subset\mathcal{D}(T')$ is a sequence such that $u_n\rightarrow u$ and $T'u_n\rightarrow v$ for some $u,v\in X$, then we obtain for all $w\in\mathcal{D}(T)$

\begin{align*}
\langle T' u_n,w\rangle=\langle u_n,Tw\rangle \rightarrow\langle u,Tw\rangle.
\end{align*}
Since $\langle T'u_n,w\rangle\rightarrow\langle v,w\rangle$, we get that $\langle u,Tw\rangle=\langle v,w\rangle$ for all $w\in\mathcal{D}(T)$ and so $u\in\mathcal{D}(T^\ast)=\mathcal{D}(T)$. As $u_n\in(\ker T)^\perp$ for all $n\in\mathbb{N}$, we also have that $u\in(\ker T)^\perp$ and so $u\in\mathcal{D}(T')$. Moreover, we obtain

\begin{align*}
\langle v,w\rangle=\langle T^\ast u,w\rangle=\langle Tu,w\rangle=\langle T'u,w\rangle,\quad w\in\mathcal{D}(T),
\end{align*}
which shows that $T'u=v$, and so $T'$ is closed.\\
Since $T'$ is moreover bijective, we obtain from the closed graph theorem that $(T')^{-1}:X\rightarrow X$ is bounded and hence $0\in\rho(T')$. Accordingly, there exists a neighbourhood of $0$ in $\mathbb{C}$ belonging entirely to the resolvent set of $T'$.\\
Now let us assume that there is some $\lambda\in(-\varepsilon,\varepsilon)\setminus\{0\}$ such that $\lambda\in\rho(T')$ and that there exists $u=u_1+u_2\in\mathcal{D}(T)=(\ker(T)\oplus\im(T))\cap\mathcal{D}(T)$ such that $\lambda u-Tu=\lambda u_1+\lambda u_2-Tu_2=0$. Then $\lambda u_1=(T-\lambda)u_2=(T'-\lambda)u_2$ and since the right hand side is in $\im(T)$ and the left hand side in $\ker(T)$, we see that both sides vanish. As $0\neq\lambda\in\rho(T')$, we get that $u_1=u_2=0$ and hence $u=0$ which shows that $\lambda\in\rho(T)$.    
\end{proof}
\noindent
Finally, let us note the following result on spectral projections for later reference.

\begin{lemma}
Let $T\in\mathcal{C}(H)$ be selfadjoint and let $\lambda_0$ be an isolated point in the spectrum $\sigma(T)$. Then $\lambda_0\in\sigma_p(T)$ and

\[\im(P_{\{\lambda_0\}})=\ker(\lambda_0-T).\]
\end{lemma}

\begin{proof}
\cite[Prop. 6.3]{Hislop}
\end{proof}


\chapter{The Gap Topology}

\section{Definition and Properties}
As before we let $H\neq\{0\}$ be a complex Hilbert space. We denote by $\mathcal{C}^\textup{sa}(H)$ the set of all densely defined, selfadjoint operators $T:\mathcal{D}(T)\subset H\rightarrow H$ and the aim of this section is to introduce a metric on the space $\mathcal{C}^\textup{sa}(H)$.\\
If $T\in\mathcal{C}^\textup{sa}(H)$, then $\pm i\notin\sigma(T)$ and we obtain in particular that $T+i$ has a bounded inverse $(T+i)^{-1}:H\rightarrow H$. As $T-i\in\mathcal{C}(H)$ and $\im((T+i)^{-1})=\mathcal{D}(T)=\mathcal{D}(T-i)$, we conclude by Lemma \ref{boundedcomposition} that 
\[\kappa(T):=(T-i)(T+i)^{-1}\in\mathcal{L}(H).\] 
The operator $\kappa(T)$ is called the \textit{Cayley-transform} of $T\in\mathcal{C}^{sa}(H)$ and we note that

\begin{align}\label{kappaidentity}
\kappa(T)=(T+i-2i)(T+i)^{-1}=I_H-2i\,(T+i)^{-1}.
\end{align}
Consequently, if $T_1,T_2\in\mathcal{C}^{sa}(H)$, then

\begin{align}\label{equivmetric}
\|\kappa(T_1)-\kappa(T_2)\|=2\|(T_1+i)^{-1}-(T_2+i)^{-1}\|
\end{align}
and $\kappa:\mathcal{C}^\textup{sa}(H)\rightarrow\mathcal{L}(H)$ is injective. 

\begin{defi}
The \textit{gap metric} on $\mathcal{C}^\textup{sa}(H)$ is defined by

\begin{align*}
d_G(T_1,T_2)=\|\kappa(T_1)-\kappa(T_2)\|.
\end{align*}
\end{defi}
\noindent
Note that $d_G$ indeed defines a metric because of the injectivity of $\kappa:\mathcal{C}^\textup{sa}(H)\rightarrow\mathcal{L}(H)$. Moreover, by \eqref{equivmetric} we obtain an equivalent metric $\delta$ by

\begin{align}\label{tildedelta}
\delta(T_1,T_2)=\|(T_1+i)^{-1}-(T_2+i)^{-1}\|.
\end{align}
Let us recall that a linear operator $U:H\rightarrow H$ is called \textit{unitary} if $U^\ast=U^{-1}$, and moreover every surjective isometry is unitary. We prove the following results along the lines of \S 1.1 of Booss-Bavnbek, Lesch and Phillips' article  \cite{UnbSpecFlow}.

\begin{theorem}\label{thm:Cayley}
If $U$ is unitary and $U-I_H$ injective, then $T:=i(I_H+U)(I_H-U)^{-1}$ is selfadjoint on $\mathcal{D}(T)=\im(I_H-U)$. Moreover, $T=i(I_H-U)^{-1}(I_H+U)$.
\end{theorem}

\begin{proof}
Since $U$ is in particular normal we see that $\ker(I_H-U^\ast)=\ker(I_H-U)$ and so

\[\overline{\im(I_H-U)}=\ker(I_H-U^\ast)^\perp=\ker(I_H-U)^\perp=H,\]
as $I_H-U$ is injective. Consequently, $\mathcal{D}(T)=\im(I_H-U)$ is dense. From

\begin{align}\label{unitarycommute}
(I_H-U)(I_H+U)=I_H-U^2=(I_H+U)(I_H-U)
\end{align} 
we obtain

\begin{align}\label{inclusionI}
\begin{split}
(I_H+U)(I_H-U)^{-1}&=(I_H-U)^{-1}(I_H-U)(I_H+U)(I_H-U)^{-1}\\
&=(I_H-U)^{-1}(I_H+U)\mid_{\im(I_H-U)}\subset(I_H-U)^{-1}(I_H+U).
\end{split}
\end{align}
On the other hand, if $u\in\mathcal{D}((I_H-U)^{-1}(I_H+U))$, then $(I_H+U)u\in\mathcal{D}((I_H-U)^{-1})=\im(I_H-U)$ and accordingly there exists $v\in H$, such that $(I_H+U)u=(I_H-U)v$. We conclude that $u=(I_H-U)v+(I_H-U)u-u$ and hence

\[u=\frac{1}{2}(I_H-U)(u+v)\in\mathcal{D}((I_H+U)(I_H-U)^{-1}).\]
We obtain from \eqref{inclusionI} that

\begin{align*}
T=i(I_H+U)(I_H-U)^{-1}=i(I_H-U)^{-1}(I_H+U).
\end{align*}
As next step, we want to show that $T$ is symmetric. If $u,v\in\mathcal{D}(T)=\im(I_H-U)$, then there exist $y,z\in H$ such that $v=y-Uy$ and $Tu=i(z+Uz)$ and we get

\begin{align*}
\langle Tu,v\rangle&=i\langle z+Uz,y-Uy\rangle=i(\langle z,y\rangle-\langle z,Uy\rangle+\langle Uz,y\rangle-\langle Uz,Uy\rangle)\\
&=i\langle Uz,y\rangle-i\langle z,Uy\rangle=\langle z-Uz,i(y+Uy)\rangle=\langle u,Tv\rangle.
\end{align*}  
Hence $T$ is symmetric and we obtain from Lemma \ref{adjointscommute}

\begin{align}\label{TinTast}
T\subset T^\ast=-i(I_H-U^\ast)^{-1}(I_H+U^\ast).
\end{align}
By arguing verbatim for $T^\ast$ as for $T$, we have

\[T^\ast=-i(I_H-U^\ast)^{-1}(I_H+U^\ast)=-i(I_H+U^\ast)(I_H-U^\ast)^{-1}\]
and so also $T^\ast$ is symmetric. Hence

\begin{align*}
T^\ast\subset T^{\ast\ast}=i(I_H-U)^{-1}(I_H+U)=T
\end{align*}
and we conclude from \eqref{TinTast} that $T=T^\ast$.
\end{proof}

\noindent
We obtain two important corollaries from the previous theorem.

\begin{cor}
If $U$ and $T$ are as in Theorem \ref{thm:Cayley}, then $\kappa(T)=U$.
\end{cor}

\begin{proof}
By Theorem \ref{thm:Cayley} we have $T=i(I_H-U)^{-1}(I_H+U)$. Hence

\begin{align*}
T+i=i(I_H-U)^{-1}(I_H+U)+i(I_H-U)^{-1}(I_H-U)=2i(I_H-U)^{-1},
\end{align*}
and so

\[(T+i)^{-1}=\frac{1}{2i}(I_H-U).\]
Analogously,

\begin{align*}
T-i=i(I_H-U)^{-1}(I_H+U)-i(I_H-U)^{-1}(I_H-U)=2i(I_H-U)^{-1}U
\end{align*}
and we obtain

\begin{align*}
\kappa(T)=(T-i)(T+i)^{-1}=(I_H-U)^{-1}U(I_H-U)=U.
\end{align*}
\end{proof}

\begin{cor}
The Cayley transform $\kappa$ induces a homeomorphism 

\[\kappa:\mathcal{C}^\textup{sa}(H)\rightarrow\{U\in\mathcal{L}(H):\, U^\ast=U^{-1}\,, \ker(U-I_H)=\{0\}\}\subset\mathcal{L}(H).\]
\end{cor}

\begin{proof}
By Theorem \ref{thm:Cayley}, we only have to show that $U:=\kappa(T)$ is unitary and $I_H-\kappa(T)$ injective for all $T\in\mathcal{C}^\textup{sa}(H)$.\\
It is clear that $U$ is surjective. For $u\in\mathcal{D}(T)$ we have

\begin{align*}
\|Tu+iu\|^2&=\langle Tu+iu,Tu+iu\rangle=\|Tu\|^2+\|u\|^2+i\langle u,Tu\rangle-i\langle Tu,u\rangle\\
&=\|Tu\|^2+\|u\|^2=\|Tu-iu\|^2
\end{align*}
and since $U(Tu+iu)=Tu-iu$, we conclude that $\|Uv\|=\|v\|$ for all $v\in H$. Hence $U$ is a surjective isometry defined on all of $H$, and consequently it is a unitary operator.\\
Now we assume that $u\in H$ is such that $\kappa(T)u=u$. Then we obtain from \eqref{kappaidentity}

\[u=\kappa(T)u=u-2i(T+i)^{-1}u\]
and hence $(T+i)^{-1}u=0$ which implies that $u=0$.
\end{proof}

\noindent
Let us recall that 

\[\kappa:\mathbb{R}\rightarrow S^1,\quad \kappa(t)=\frac{t-i}{t+i}\]
induces a homeomorphism onto $S^1\setminus\{1\}\subset\mathbb{C}$.

\begin{lemma}\label{cayleyspecformula}
If $T\in\mathcal{C}^{sa}(H)$, then 

\begin{align*}
\lambda-T=(\lambda+i)(\kappa(\lambda)-\kappa(T))(I_H-\kappa(T))^{-1}.
\end{align*}
\end{lemma}

\begin{proof}
This follows from 

\begin{align*}
\lambda-T&=\lambda-i(I_H+\kappa(T))(I_H-\kappa(T))^{-1}=(\lambda(I_H-\kappa(T))-i(I_H+\kappa(T)))(I_H-\kappa(T))^{-1}\\
&=(\lambda-\lambda\kappa(T)-i-i\kappa(T))(I_H-\kappa(T))^{-1}=((\lambda-i)-(\lambda+i)\kappa(T))(I_H-\kappa(T))^{-1}\\
&=(\lambda+i)((\lambda-i)(\lambda+i)^{-1}-\kappa(T))(I_H-\kappa(T))^{-1}=(\lambda+i)(\kappa(\lambda)-\kappa(T))(I_H-\kappa(T))^{-1}.
\end{align*}
\end{proof}

\noindent
As a consequence we obtain the following important corollary illustrating the relation between the spectrum of an operator in $\mathcal{C}^\textup{sa}(H)$ and the spectrum of its Cayley transform.

\begin{cor}\label{cor:spectralmain}
If $T\in\mathcal{C}^\textup{sa}(H)$, then

\begin{enumerate}
	\item[(i)] $\ker(\lambda-T)\neq \{0\}$ if and only if $\ker(\kappa(\lambda)-\kappa(T))\neq\{0\}$. Moreover the dimensions of both spaces coincide.
	\item[(ii)] $\im(\lambda-T)=\im(\kappa(\lambda)-\kappa(T))$. 
\end{enumerate}
Moreover, 

\begin{itemize}
	\item $\lambda\in\rho(T)\Longleftrightarrow\kappa(\lambda)\in\rho(\kappa(T))$, 
	\item $\lambda\in\sigma(T)\Longleftrightarrow\kappa(\lambda)\in\sigma(\kappa(T))$,
	\item $\lambda\in\sigma_p(T)\Longleftrightarrow\kappa(\lambda)\in\sigma_p(\kappa(T))$, 
	\item $\lambda\in\sigma_{ess}(T)\Longleftrightarrow\kappa(\lambda)\in\sigma_{ess}(\kappa(T))$.
\end{itemize}
\end{cor}

\begin{proof}
By the previous Lemma \ref{cayleyspecformula} we know that

\begin{align*}
\lambda-T=(\lambda+i)(\kappa(\lambda)-\kappa(T))(I_H-\kappa(T))^{-1}
\end{align*}
and moreover $(I_H-\kappa(T))^{-1}$ maps $\mathcal{D}(T)$ bijectively onto $H$. This implies the assertions on $\ker(\lambda-T)$ and $\im(\lambda-T)$. The remaining part of the corollary is an immediate consequence of them.
\end{proof}

\begin{lemma}\label{1inspec}
Let $T\in\mathcal{C}^\textup{sa}(H)$. Then

\begin{enumerate}
	\item[(i)] $1\in\rho(\kappa(T))\Longleftrightarrow\mathcal{D}(T)=H$, and this is true if and only if $T$ is bounded.
	\item[(ii)] $1\in\sigma_{ess}(\kappa(T))\Longleftrightarrow\mathcal{D}(T)\neq H$, and this is true if and only if $T$ is unbounded.  
\end{enumerate}
\end{lemma}

\begin{proof}
The assertions regarding the boundedness and unboundedness of $T$ follow from Lemma \ref{boundedclosed} and the assumption that $T$ is densely defined.\\
By \eqref{kappaidentity} we have

\begin{align*}
I_H-\kappa(T)=2i(T+i)^{-1}\in\mathcal{L}(H)
\end{align*}
mapping $H$ bijectively onto $\mathcal{D}(T)$. Accordingly, if $1\in\rho(\kappa(T))$, we infer $H=\im(I_H-\kappa(T))=\mathcal{D}(T)$. Conversely, if $\mathcal{D}(T)=H$, then $I_H-\kappa(T)$ maps $H$ bijectively onto $H$ showing that $1\in\rho(\kappa(T))$ by the closed graph theorem (Theorem \ref{closedgraph}). Hence assertion $(i)$ is proved.\\
In order to show $(ii)$ we note at first that by $(i)$, $1\in\sigma(\kappa(T))$ if and only if $\mathcal{D}(T)\neq H$. Now it remains to show that if $1\in\sigma(\kappa(T))$, then we actually have $1\in\sigma_{ess}(\kappa(T))$. But, if $\mathcal{D}(T)\neq H$, we see that $\im(I_H-\kappa(T))=\mathcal{D}(T)$ is a proper dense subspace of $H$ and hence in particular not closed. Accordingly, $I_H-\kappa(T)$ is not a Fredholm operator.
\end{proof}
\noindent
We obtain from Corollary \ref{cor:spectralmain}:

\begin{cor}\label{cor:spectralmainII}
If $T\in\mathcal{C}^\textup{sa}(H)$, then

\begin{itemize}
	\item[(i)] $\kappa(\sigma(T))=\sigma(\kappa(T))$ if $T$ is bounded.
	\item[(ii)] $\kappa(\sigma(T))\cup\{1\}=\sigma(\kappa(T))$ if $T$ is unbounded. 
\end{itemize}
\end{cor}


\section{Stability of Spectra}\label{sec:stability}

We recall at first the spectral stability for bounded operators:

\begin{theorem}\label{stabspecbounded}
Let $E$ be a Banach space, $A\in\mathcal{L}(E)$ and $\Omega\subset\mathbb{C}$ an open neighbourhood of $\sigma(A)$. Then there exists $\varepsilon>0$ such that $\sigma(B)\subset\Omega$ for any $B\in\mathcal{L}(E)$ with $\|A-B\|<\varepsilon$. Moreover, the same conclusion holds true if we replace $\sigma$ by $\sigma_{ess}$.
\end{theorem}

\begin{proof}
The spectral stability holds for any complex unital Banach algebra $R$ (cf. e.g. \cite[96.5]{HeuserFunkAna}), and so the first assertion follows by setting $R=\mathcal{L}(E)$. For the stability of the essential spectrum, we let $R$ be the Calkin algebra $\calkin(E)=\mathcal{L}(E)/\mathcal{K}(E)$ of $E$ and use that the quotient map $q:\mathcal{L}(E)\rightarrow\calkin(E)$ is continuous.
\end{proof}
\noindent
The next two results can be found in \cite[\S 2.2]{UnbSpecFlow}.

\begin{lemma}\label{gapstabspec}
Let $K\subset\mathbb{C}$ be compact. Then 

\[\{T\in\mathcal{C}^\textup{sa}(H):\,K\subset\rho(T)\}\]
and 

\[\{T\in\mathcal{C}^\textup{sa}(H):\,K\subset\rho_{ess}(T)\}\]
are open in $\mathcal{C}^\textup{sa}(H)$, where $\rho_{ess}(T):=\mathbb{C}\setminus\sigma_{ess}(T)$. 
\end{lemma}

\begin{proof}
We denote by $\mathcal{U}(H)$ the set of all unitary operators acting on $H$. By using Corollary \ref{cor:spectralmain} and Corollary \ref{cor:spectralmainII}, we have

\begin{align*}
\{T\in\mathcal{C}^\textup{sa}(H):\, K\subset\rho(T)\}&=\{T\in\mathcal{C}^\textup{sa}(H):\,\sigma(T)\subset(\mathbb{C}\setminus K)\cap\mathbb{R}\}\\
&=\{T\in\mathcal{C}^\textup{sa}(H):\sigma(\kappa(T))\subset\kappa((\mathbb{C}\setminus K)\cap\mathbb{R})\cup\{1\}\}. 
\end{align*}
We conclude from Theorem \ref{thm:Cayley} that 

\[\{T\in\mathcal{C}^\textup{sa}(H):\, K\subset\rho(T)\}=\kappa^{-1}(\{U\in\mathcal{U}(H):\,\sigma(U)\subset\kappa((\mathbb{C}\setminus K)\cap\mathbb{R})\cup\{1\}\}).\]
As $K$ is compact, the set $\kappa((\mathbb{C}\setminus K)\cap\mathbb{R})\cup\{1\}$ is open and we see from Theorem \ref{stabspecbounded} that
\[\{U\in\mathcal{U}(H):\,\sigma(U)\subset \kappa((\mathbb{C}\setminus K)\cap\mathbb{R})\cup\{1\}\}\]
is open in $\mathcal{L}(H)$. Since $\kappa$ is a homeomorphism by Theorem \ref{thm:Cayley} , we obtain the assertion. The proof for $\rho_{ess}(T)$ proceeds along the same lines.   
\end{proof}

\begin{theorem}\label{thm:rescont}
Let $\emptyset\neq K\subset\mathbb{C}$ and $\Omega_K=\{T\in\mathcal{C}^\textup{sa}(H):\,K\subset\rho(T)\}$. Then the map

\begin{align*}
R:K\times\Omega_K\rightarrow\mathcal{L}(H),\quad(\lambda,T)\mapsto(T-\lambda)^{-1}
\end{align*}
is continuous.
\end{theorem}

\begin{proof}
We fix $z_0\in K$ and note that we have for $(\lambda,T)\in K\times\Omega_K$

\begin{align}\label{ClosedOps-align-theoremrescontres}
\begin{split}
R(\lambda,T)&=(T-\lambda)^{-1}=((T-z_0)-(\lambda-z_0))^{-1}\\
&=((T-z_0)(I_H-(\lambda-z_0)(T-z_0)^{-1}))^{-1}\\
&=(I_H-(\lambda-z_0)(T-z_0)^{-1})^{-1}(T-z_0)^{-1}=F(G(\lambda,T)),
\end{split}
\end{align}
where the maps $F$ and $G$ are defined by\footnote{For $A\subset\mathbb{C}\setminus\{0\}$, we denote $A^{-1}=\{\frac{1}{z}\in\mathbb{C}: z\in A\}$.}

\begin{align*}
G:K&\times\Omega_K\rightarrow K\times\{S\in\mathcal{L}(H):(K-z_0)^{-1}\subset\rho(S)\},\\
&(\lambda,T)\mapsto(\lambda,(T-z_0)^{-1})
\end{align*}
and

\begin{align*}
F:&K\times\{S\in\mathcal{L}(H):(K-z_0)^{-1}\subset\rho(S)\}\rightarrow\mathcal{L}(H),\\
&(\lambda,S)\mapsto(I_H-(\lambda-z_0)S)^{-1}S,
\end{align*}
respectively. $G$ is continuous as $d_G$ is equivalent to the metric \eqref{tildedelta}. Furthermore, the continuity of $F$ follows by a simple computation using the continuity of the inversion on $GL(H)$ (cf. \cite[I.(4.24),III.3.1]{Kato}).
\end{proof}

\noindent
As an important corollary of the previous theorem, we obtain the continuity of the spectral projections which we introduced in Theorem \ref{spectralprojections}.

\begin{cor}\label{ClosedOps-cor-contproj}
Let $\bigtriangleup\subset\mathbb{C}$ be a Cauchy domain with boundary $\Gamma$ and

\begin{align*}
\Omega_\Gamma=\{T\in\mathcal{C}^\textup{sa}(H):\Gamma\subset\rho(T)\}\subset\mathcal{C}^\textup{sa}(H).
\end{align*}
Then the map

\begin{align*}
\Omega&_\Gamma\rightarrow\mathcal{L}(H),\quad T\mapsto P_\Gamma(T):=\frac{1}{2\pi i}\int_\Gamma{(\lambda-T)^{-1}\,d\lambda}
\end{align*}
is continuous.
\end{cor}

\begin{proof}
For any $T,S\in\Omega_{\Gamma}$ we have by \cite[(97.4)]{HeuserFunkAna}

\begin{align}\label{ClosedOps-align-contproj}
\|P_\Gamma(T)-P_\Gamma(S)\|\leq\frac{1}{2\pi}|\Gamma|\max_{\lambda\in\Gamma}\|(\lambda-T)^{-1}-(\lambda-S)^{-1}\|,
\end{align}
where $|\Gamma|$ denotes the length of the contour $\Gamma$. The rest of the proof is a standard argument in analysis:\\
Let $T\in\Omega_\Gamma$ and $\varepsilon>0$. By Theorem \ref{thm:rescont} there exists $\delta(\lambda')>0$ for any $\lambda'\in\Gamma$ such that 

\begin{align*}
\|(\lambda-S)^{-1}-(\lambda'-T)^{-1}\|<\frac{\pi\varepsilon}{|\Gamma|}
\end{align*}
if 
 
\begin{align*}
\lambda\in U(\lambda',\delta(\lambda')):=\{\lambda\in\Gamma:|\lambda-\lambda'|<\delta(\lambda')\}\quad\text{and}\quad d_G(S,T)<\delta(\lambda').
\end{align*}
As $\Gamma$ is compact we can find $\lambda_1,\ldots,\lambda_n\in\Gamma$ such that $\bigcup^n_{i=1}{U(\lambda_i,\delta(\lambda_i))}=\Gamma$, and we set $\delta:=\min_{1\leq i\leq n}\delta(\lambda_i)$.\\
Now, for any $\lambda\in\Gamma$ there exists $1\leq i\leq n$ such that $\lambda\in U(\lambda_i,\delta(\lambda_i))$ and hence we obtain for $S\in\Omega_\Gamma$ such that $d_G(S,T)<\delta$,

\begin{align*}
\|(\lambda-T)^{-1}-(\lambda-S)^{-1}\|&\leq \|(\lambda-T)^{-1}-(\lambda_i-T)^{-1}\|+\|(\lambda_i-T)^{-1}-(\lambda-S)^{-1}\|<\frac{2\pi\varepsilon}{|\Gamma|}.
\end{align*}
We get by \eqref{ClosedOps-align-contproj}

\begin{align*}
\|P_\Gamma(T)-P_\Gamma(S)\|\leq\frac{1}{2\pi}|\Gamma|\max_{\lambda\in\Gamma}\|(\lambda-T)^{-1}-(\lambda-S)^{-1}\|<\varepsilon
\end{align*}
for all $S\in\Omega_\Gamma$ such that $d_G(S,T)<\delta$.
\end{proof}







\section{Spaces of Selfadjoint Fredholm Operators}\label{section-topology}

In this section, we briefly consider different topologies on spaces of selfadjoint Fredholm operators following mainly Lesch's work \cite{LeschSpecFlowUniqu}. We set

\[\mathcal{CF}^\textup{sa}(H)=\{ T\in\mathcal{C}^\textup{sa}(H):\, T\,\text{ Fredholm}\}\]
and recall that $\mathcal{CF}^\textup{sa}(H)$ is a metric space with respect to the \textit{gap metric} 

\[d_G(T_1,T_2)=\|\kappa(T_1)-\kappa(T_2)\|,\quad T_1,T_2\in\mathcal{CF}^\textup{sa}(H).\] 
For $T\in\mathcal{C}^\textup{sa}(H)$ one can show that $I_H+T^2$ is invertible and selfadjoint on $\mathcal{D}(T^2)\subset\mathcal{D}(T)$. We obtain from Lemma \ref{boundedcomposition} that $F(T):=T(I_H+T^2)^{-\frac{1}{2}}\in\mathcal{L}(H)$. The \textit{Riesz metric} on $\mathcal{CF}^\textup{sa}(H)$ is defined by

\[d_R(T_1,T_2)=\|F(T_1)-F(T_2)\|,\quad T_1,T_2\in\mathcal{CF}^\textup{sa}(H).\]
We set
\[\mathcal{BF}^\textup{sa}(H)=\{T\in\mathcal{CF}^{sa}(H):\,\mathcal{D}(T)=H\}\]
and note that we have a metric on $\mathcal{BF}^\textup{sa}(H)$ given by

\[d_N(T_1,T_2)=\|T_1-T_2\|,\quad T_1,T_2\in\mathcal{BF}^\textup{sa}(H).\]
The following result is due to Nicolaescu \cite{Nicolaescu}.

\begin{theorem}
The topology induced by the Riesz metric on $\mathcal{CF}^\textup{sa}(H)$ is strictly finer than the topology induced by the gap metric.
\end{theorem}

\noindent
Note that the theorem implies in particular that every path in $\mathcal{CF}^\textup{sa}(H)$ which is continuous with respect to the Riesz metric is also continuous with respect to the gap metric. Let us now consider $\mathcal{BF}^\textup{sa}(H)$ and let us note a classical result by Cordes and Labrousse \cite{Cordes}:

\begin{theorem}
On $\mathcal{BF}^\textup{sa}(H)$ the topologies induced by $d_N,d_R$ and $d_G$ coincide.
\end{theorem}

\noindent
The space $\mathcal{BF}^\textup{sa}(H)$ was investigated by Atiyah and Singer in \cite{AtiyahSinger} and it plays a fundamental role in topology.

\begin{theorem}
The space $\mathcal{BF}^\textup{sa}(H)$ consists of three connected components

\begin{align*}
\mathcal{BF}^\textup{sa}_+(H)&=\{T\in\mathcal{BF}^\textup{sa}(H):\,\sigma_{ess}(T)\subset(0,\infty)\}\\
\mathcal{BF}^\textup{sa}_-(H)&=\{T\in\mathcal{BF}^\textup{sa}(H):\,\sigma_{ess}(T)\subset(-\infty,0)\}
\end{align*}
and
\[\mathcal{BF}^\textup{sa}_\ast(H)=\mathcal{BF}^\textup{sa}(H)\setminus(\mathcal{BF}^\textup{sa}_+(H)\cup\mathcal{BF}^\textup{sa}_-(H)).\]
The spaces $\mathcal{BF}^\textup{sa}_+(H)$ and $\mathcal{BF}^\textup{sa}_-(H)$ are contractible, whereas for $k\in\mathbb{N}$

\begin{align*}
\pi_k(\mathcal{BF}^\textup{sa}_\ast(H))=
\begin{cases}
0,\quad\text{if}\, k\,\text{even}\\
\mathbb{Z},\quad\text{if}\, k\,\text{odd}
\end{cases}.
\end{align*}
\end{theorem}

\noindent
Here $\pi_k(\mathcal{BF}^\textup{sa}_\ast(H))$, $k\in\mathbb{N}$, denote the homotopy groups of the space $\mathcal{BF}^\textup{sa}_\ast(H)$. It was shown by Lesch in \cite{LeschSpecFlowUniqu} that the assertions of the previous theorem also hold for $\mathcal{CF}^\textup{sa}(H)$ with respect to the Riesz metric. He also proved the following surprising theorem, which is in strong contrast to the previously mentioned results:

\begin{theorem}
The space $\mathcal{CF}^\textup{sa}(H)$ is connected with respect to the gap metric.
\end{theorem}


\chapter{The Spectral Flow} 

\section{Definition of the Spectral Flow}
The aim of this section is to introduce the spectral flow along the lines of \cite[\S 2.2]{UnbSpecFlow}.
Before we begin the construction, we rephrase the results about the spectral stability from Section \ref{sec:stability}. In what follows, we denote for $T\in\mathcal{CF}^\textup{sa}(H)$ such that $a,b\notin\sigma(T)$ by

\[\chi_{[a,b]}(T)=\frac{1}{2\pi i}\int_\Gamma{(\lambda-T)^{-1}\,d\lambda}\]
the spectral projection with respect to the interval $[a,b]$, where $\Gamma$ is the circle of radius $\frac{b-a}{2}$ around $\frac{a+b}{2}$. Moreover, we will need the following well known fact.

\begin{lemma}\label{projectionsclose}
Let $P,Q\in\mathcal{L}(E)$ be two projections on the Banach space $E$. If $\|P-Q\|<1$, then 

\[\dim\im(P)=\dim\im(Q).\]
\end{lemma}

\begin{proof}
\cite[Lemma II.4.3]{GohbergClasses}
\end{proof}
\noindent
The next lemma is a cornerstone in the definition of the spectral flow.

\begin{lemma}\label{lem:sflneighbourhood}
Let $T_0\in\mathcal{CF}^{sa}(H)$ be given.

\begin{enumerate}
	\item[(i)] There exists a positive real number $a$ such that $\pm a\notin\sigma(T_0)$, and an open neighbourhood $N\subset\mathcal{CF}^\textup{sa}(H)$ of $T_0$ such that
	
	\begin{align}\label{specproj}
	N\rightarrow\mathcal{L}(H),\quad T\mapsto\chi_{[-a,a]}(T)
	\end{align}
is continuous and, moreover, $\sigma_{ess}(T)\cap[-a,a]=\emptyset$ and the projection $\chi_{[-a,a]}(T)$ has finite rank for all $T\in N$.
\item[(ii)] If $-a\leq c<d\leq a$ are such that $c,d\in\rho(T)$ for all $T\in N$, then $T\mapsto\chi_{[c,d]}(T)$ is continuous on $N$. Moreover, the rank of $\chi_{[c,d]}(T)$, $T\in N$, is finite and constant on each connected component of $N$. 
\end{enumerate}
\end{lemma}

\begin{proof}
By Lemma \ref{spectralFredSelf} there exists $a>0$ such that $[-a,a]\cap\sigma(T_0)\subset\{0\}$. Now

\[N:=\{T\in\mathcal{CF}^\textup{sa}(H):\, [-a,a]\subset\rho_{ess}(T_0),\,\,\pm a\notin\sigma(T_0)\}\]
is open with respect to the gap metric by Lemma \ref{gapstabspec} and the map \eqref{specproj} is continuous by Corollary \ref{ClosedOps-cor-contproj}. Moreover, $\chi_{[-a,a]}(T)$ has finite rank for all $T\in N$ as $[-a,a]\cap\sigma_{ess}(T)=\emptyset$. Hence we have shown the first assertion. The second assertion now follows immediately from Corollary \ref{ClosedOps-cor-contproj}.  
\end{proof}

\noindent
Note that if $c,d$ and $N$ are as in (ii) of the previous lemma, then

\[\im(\chi_{[c,d]}(T))=\bigoplus_{\mu\in[c,d]}{\ker(\mu-T)}\subset H\]
for all $T\in N$.\\
\noindent
Let now $\mathcal{A}:I\rightarrow\mathcal{CF}^{sa}(H)$ be a continuous path with respect to the gap topology, where we denote by $I=[0,1]$ the unit interval. By the previous Lemma \ref{lem:sflneighbourhood} we conclude that for every $\lambda\in I$ there exists $a>0$ and an open neighbourhood $N_{\lambda,a}\subset\mathcal{CF}^\textup{sa}(H)$ such that $\pm a\in\rho(T)$ for all $T\in N_{\lambda,a}$ and the map

\begin{align*}
N_{\lambda,a}\rightarrow\mathcal{L}(H),\quad T\mapsto\chi_{[-a,a]}(T)
\end{align*} 
is continuous. Moreover, all $\chi_{[-a,a]}(T)$, $T\in N_{\lambda,a}$, have the same finite rank. Now the counterimages of the $N_{\lambda,a}$ under $\mathcal{A}$ define an open covering of the unit interval and, by using the Lebesgue number of this covering, we can find $0=\lambda_0\leq \lambda_1\leq\ldots\leq \lambda_n=1$ and $a_i>0$, $i=1,\ldots n$, such that the maps

\begin{align*}
[\lambda_{i-1},\lambda_i]\ni\lambda\mapsto\chi_{[-a_i,a_i]}(\mathcal{A}_\lambda)\in\mathcal{L}(H)
\end{align*}
are continuous and of constant rank. In what follows we denote for $[c,d]\subset[-a_i,a_i]$ by

\[E_{[c,d]}(\mathcal{A}_\lambda)=\bigoplus_{\mu\in[c,d]}{\ker(\mu-\mathcal{A}_\lambda)}\]
the direct sum of the eigenspaces with respect to eigenvalues in the interval $[c,d]$.
Our intention is to define the spectral flow of $\mathcal{A}:I\rightarrow\mathcal{CF}^{sa}(H)$ by

\begin{align}\label{IndPre-align-specflow}
\sfl(\mathcal{A})=\sum^n_{i=1}{\left(\dim E_{[0,a_i]}(\mathcal{A}_{\lambda_i})-\dim E_{[0,a_i]}(\mathcal{A}_{\lambda_{i-1}})\right)}
\end{align}
but we need to prove its well definedness before.

\begin{lemma}
$\sfl(\mathcal{A})$ depends only on the continuous map $\mathcal{A}$.
\end{lemma}

\begin{proof}
We decompose the proof into three steps.\\
Let us begin by considering $\lambda_0,\ldots,\lambda_n$ and $a_1,\ldots,a_n$ as in \eqref{IndPre-align-specflow}, and we take a further instant $\lambda^\ast\in(0,1)$ such that $\lambda_{i-1}<\lambda^\ast<\lambda_i$ for some $i$. If we now use the two maps

\begin{align*}
[\lambda_{i-1},\lambda^\ast]&\ni\lambda\mapsto\chi_{[-a_i,a_i]}(\mathcal{A}_\lambda),\qquad [\lambda^\ast,\lambda_i]\ni\lambda\mapsto\chi_{[-a_i,a_i]}(\mathcal{A}_\lambda)
\end{align*} 
instead of

\begin{align*}
[\lambda_{i-1},\lambda_i]\ni\lambda\mapsto\chi_{[-a_i,a_i]}(\mathcal{A}_\lambda)
\end{align*}
for the computation of \eqref{IndPre-align-specflow}, then the sum does not change as the two new appearing terms cancel each other.\\
In the second step we do not change the partition of the interval but instead the numbers $a_i$. Let $[c,d]\subset[0,1]$ be any subinterval and 

\[\lambda\mapsto\chi_{[-a_1,a_1]}(\mathcal{A}_\lambda),\qquad\lambda\mapsto\chi_{[-a_2,a_2]}(\mathcal{A}_\lambda)\]
two continuous maps as in \eqref{IndPre-align-specflow} which are defined and continuous on $[c,d]$. We may assume without loss of generality that $a_1$ is greater or equal to $a_2$. As $a_1,a_2\notin\sigma(\mathcal{A}_\lambda)$ for all $\lambda\in[c,d]$, we obtain by Theorem \ref{spectralprojections}

\begin{align*} 
\dim E_{[0,a_1]}(\mathcal{A}_\lambda)-\dim E_{[0,a_2]}(\mathcal{A}_\lambda)=\dim\im(\chi_{[a_2,a_1]}(\mathcal{A}_\lambda))
\end{align*}
which is a constant function on $[c,d]$ by Lemma \ref{lem:sflneighbourhood} (ii). Consequently,

\begin{align*}
\dim E_{[0,a_1]}(\mathcal{A}_d)-\dim E_{[0,a_1]}(\mathcal{A}_c)&=(\dim E_{[0,a_2]}(\mathcal{A}_d)+\dim\im(\chi_{[a_2,a_1]}(\mathcal{A}_d))\\
&-(\dim E_{[0,a_2]}(\mathcal{A}_c)+\dim\im(\chi_{[a_2,a_1]}(\mathcal{A}_c)))\\
&=\dim E_{[0,a_2]}(\mathcal{A}_d)-\dim E_{[0,a_2]}(\mathcal{A}_c).
\end{align*}
Finally, let us consider the general case, i.e. we have two partitions $\lambda_0,\ldots,\lambda_n$ and $\lambda'_0,\ldots,\lambda'_m$ having associated numbers $a_1,\ldots,a_n$ and $a'_1,\ldots,a'_m$, respectively, as in \eqref{IndPre-align-specflow}. The union of both partitions yield a third one $\lambda''_0,\ldots,\lambda''_{m+n}$, which is finer than $\lambda_0,\ldots,\lambda_{n}$ and $\lambda'_1,\ldots,\lambda'_{m}$. By our first step of the proof we obtain

\begin{align}\label{proofwelldef}
\begin{split}
&\sum^{n}_{i=1}{\left(\dim E_{[0,a_i]}(\mathcal{A}_{\lambda_i})-\dim E_{[0,a_i]}(\mathcal{A}_{\lambda_{i-1}})\right)}=\sum^{m+n}_{i=1}{\left(\dim E_{[0,b_i]}(\mathcal{A}_{\lambda''_i})-\dim E_{[0,b_i]}(\mathcal{A}_{\lambda''_{i-1}})\right)}\\
&\sum^m_{i=1}{\left(\dim E_{[0,a'_i]}(\mathcal{A}_{\lambda'_i})-\dim E_{[0,a'_i]}(\mathcal{A}_{\lambda'_{i-1}})\right)}=\sum^{m+n}_{i=1}{\left(\dim E_{[0,b'_i]}(\mathcal{A}_{\lambda''_i})-\dim E_{[0,b'_i]}(\mathcal{A}_{\lambda''_{i-1}})\right)},
\end{split}
\end{align}
for suitable $b_1,\ldots,b_{m+n}\in \{a_1,\ldots,a_n\}$ and $b'_1,\ldots,b'_{m+n}\in\{a'_1,\ldots a'_{m}\}$. Now the same partition is used on the right hand sides in \eqref{proofwelldef}, and so these sums are equal by the second step of our proof.
\end{proof}

\noindent
Let us point out that we have defined the spectral flow for paths that are parametrised by the unit interval only for the sake of simplicity of notation. Clearly, the same formula as \eqref{IndPre-align-specflow} can be used to define $\sfl(\mathcal{A})$ for paths $\mathcal{A}:[a,b]\rightarrow\mathcal{CF}^\textup{sa}(H)$, where $[a,b]\subset\mathbb{R}$ is a compact interval. In what follows, we will use this without further mention.


\section{Properties and Uniqueness}\label{section-properties}

In this section we introduce the basic properties of the spectral flow, where we follow essentially \cite{Phillips}. We begin with some simple observations that all follow from its construction.

\begin{lemma}\label{IndPre-lemma-N}
Let $N\subset\mathcal{CF}^{sa}(H)$ be a neighbourhood as in the construction of the spectral flow, i.e. there exists $a>0$ such that $-a,a\in\rho(T)$, $[-a,a]\cap\sigma_{ess}(T)=\emptyset$ for all $T\in N$, the map

\begin{align*}
N\ni T\mapsto \chi_{[-a,a]}(T)\in\mathcal{L}(H)
\end{align*}
is continuous and all $\chi_{[-a,a]}(T)$ have the same finite rank.\\
If $\mathcal{A}^1,\mathcal{A}^2:I\rightarrow\mathcal{CF}^{sa}(H)$ are gap continuous and

\begin{align}\label{neighbourhood}
\mathcal{A}^1(I),\mathcal{A}^2(I)\subset N,\quad \mathcal{A}^1_0=\mathcal{A}^2_0,\quad \mathcal{A}^1_1=\mathcal{A}^2_1,
\end{align} 
then 

\[\sfl(\mathcal{A}^1)=\sfl(\mathcal{A}^2).\]
\end{lemma}

\begin{proof}
We obtain from \eqref{neighbourhood} and the definition \eqref{IndPre-align-specflow}

\begin{align*}
\sfl(\mathcal{A}^1)&=\dim E_{[0,a]}(\mathcal{A}^1_1)-\dim E_{[0,a]}(\mathcal{A}^1_0)\\
&=\dim E_{[0,a]}(\mathcal{A}^2_1)-\dim E_{[0,a]}(\mathcal{A}^2_0)=\sfl(\mathcal{A}^2).
\end{align*}
\end{proof}

\begin{lemma}\label{IndPre-lemma-sflbasicprop}
\begin{enumerate}
\item[(i)] If $\mathcal{A}^1,\mathcal{A}^2:I\rightarrow\mathcal{CF}^{sa}(H)$ are two gap continuous paths such that $\mathcal{A}^2_0=\mathcal{A}^1_1$, then

	\begin{align*}
	\sfl(\mathcal{A}^1\ast\mathcal{A}^2)=\sfl(\mathcal{A}^1)+\sfl(\mathcal{A}^2).
	\end{align*}
\item[(ii)] If $\mathcal{A}:I\rightarrow\mathcal{CF}^{sa}(H)$ is gap continuous and $\mathcal{A}'$ is defined by $\mathcal{A}'_t=\mathcal{A}_{1-t}$, then

\begin{align*}
\sfl(\mathcal{A}')=-\sfl(\mathcal{A}).
\end{align*}

\item[(iii)] If $\mathcal{A}:I\rightarrow \mathcal{CF}^{sa}(H)$ is gap continuous and $\mathcal{A}_t$ is invertible for all $t\in I$, then $\sfl(\mathcal{A})=0$. 
\end{enumerate}
\end{lemma}

\begin{proof}
The first two assertions follow immediately from the definition \eqref{IndPre-align-specflow}. For the third assertion we just have to observe that by Lemma \ref{gapstabspec} we can find $\delta>0$ such that $\sigma(\mathcal{A}_\lambda)\cap[-\delta,\delta]=\emptyset$ for all $\lambda\in[0,1]$. Then

\begin{align*}
\sfl(\mathcal{A})=\dim E_{[0,\delta]}(\mathcal{A}_1)-\dim E_{[0,\delta]}(\mathcal{A}_0)=0.
\end{align*}
\end{proof}

\noindent
In what follows we set

\begin{align*}
G\mathcal{C}^\textup{sa}(H)=\{T\in\mathcal{C}^\textup{sa}(H):\, T\,\text{invertible}\,\}.
\end{align*}
The probably most important property of the spectral flow is its homotopy invariance which we prove in the following lemma.

\begin{lemma}\label{IndPre-lemma-sflhominv}
Let $h:I\times I\rightarrow\mathcal{CF}^{sa}(H)$ be a continuous map such that
 
\begin{align*}
h(I\times\partial I)\subset G\mathcal{C}^{sa}(H).
\end{align*}
Then

\begin{align*}
\sfl(h(0,\cdot))=\sfl(h(1,\cdot)).
\end{align*}

\end{lemma}

\begin{proof}
As $h(I\times I)\subset\mathcal{CF}^{sa}(H)$ is compact, we can find a finite open covering

\begin{align*}
h(I\times I)\subset\bigcup^n_{i=1}{N_i},
\end{align*}
where $N_i\subset\mathcal{CF}^\textup{sa}(H)$, $i=1,\ldots,n$, are open sets as in the construction of the spectral flow. Accordingly, for each $N_i$ there exists $a_i>0$ such that $-a_i,a_i\in\rho(T)$, $[-a_i,a_i]\cap\sigma_{ess}(T)=\emptyset$ for all $T\in N_i$, the map 

\begin{align*}
N_i\ni T\mapsto \chi_{[-a_i,a_i]}(T)\in\mathcal{L}(H)
\end{align*} 
is continuous and all $\chi_{[-a_i,a_i]}(T)$ are projections of the same finite rank.\\
If $\varepsilon_0>0$ is a Lebesgue number of the open covering

\begin{align*}
I\times I=\bigcup^n_{i=1}{h^{-1}(N_i)},
\end{align*} 
then each subset of $I\times I$ of diameter less than $\varepsilon_0$ is entirely contained in one of the $h^{-1}(N_i)$.\\
We choose numbers $0=\lambda_0<\lambda_1<\ldots<\lambda_m=1$ such that $|\lambda_i-\lambda_{i-1}|<\frac{\varepsilon_0}{\sqrt{2}}$, $1\leq i\leq m$. Then for each pair $1\leq i,j\leq m$, $h([\lambda_{i-1},\lambda_i]\times[\lambda_{j-1},\lambda_j])$ is contained in one of the $N_k$. Now we obtain for any $h\mid_{[\lambda_{i-1},\lambda_i]\times[\lambda_{j-1},\lambda_j]}$ from Lemma \ref{IndPre-lemma-N} and Lemma \ref{IndPre-lemma-sflbasicprop}

\begin{align*}
\sfl(h(\lambda_{i-1},\cdot)\mid_{[\lambda_{j-1},\lambda_j]})&=\sfl(h(\cdot,\lambda_{j-1})\mid_{[\lambda_{i-1},\lambda_i]})+\sfl(h(\lambda_i,\cdot)\mid_{[\lambda_{j-1},\lambda_j]})\\
&-\sfl(h(\cdot,\lambda_j)\mid_{[\lambda_{i-1},\lambda_i]}).
\end{align*}
Moreover,

\begin{align*}
\sfl(h(\cdot,0)\mid_{[\lambda_{i-1},\lambda_i]})=\sfl(h(\cdot,1)\mid_{[\lambda_{i-1},\lambda_i]})=0,\quad i=1,\ldots,m,
\end{align*}
by the third part of Lemma \ref{IndPre-lemma-sflbasicprop}. By using the first part of Lemma \ref{IndPre-lemma-sflbasicprop} once again, we obtain

\begin{align*}
\sfl(h(\lambda_{i-1},\cdot))&=\sum^m_{j=1}{\sfl(h(\lambda_{i-1},\cdot)\mid_{[\lambda_{j-1},\lambda_j]})}\\
&=\sum^m_{j=1}{\sfl(h(\cdot,\lambda_{j-1})\mid_{[\lambda_{i-1},\lambda_i]})+\sfl(h(\lambda_i,\cdot)\mid_{[\lambda_{j-1},\lambda_j]})-\sfl(h(\cdot,\lambda_j)\mid_{[\lambda_{i-1},\lambda_i]})}\\
&=\sum^m_{j=1}{\sfl(h(\lambda_{i},\cdot)\mid_{[\lambda_{j-1},\lambda_j]})}=\sfl(h(\lambda_{i},\cdot)).
\end{align*}
Hence

\begin{align*}
\sfl(h(0,\cdot))=\sfl(h(\lambda_0,\cdot))=\sfl(h(\lambda_1,\cdot))=\sfl(h(1,\cdot)).
\end{align*}
\end{proof}

\noindent
We now want to discuss a normalisation property, which is needed for the uniqueness of the spectral flow. Let $\{e_k\}_{k\in\mathbb{Z}}$ be a complete orthonormal system of the Hilbert space $H$, which we now assume to be separable. Denote by $P_+$ the orthogonal projection onto the closure of the space spanned by $\{e_k\}_{k\in\mathbb{N}}$, by $P_-$ the orthogonal projection onto the closure of the space spanned by $\{e_{-k}\}_{k\in\mathbb{N}}$, and by $P_0$ the orthogonal projection onto the span of $e_0$. Then 

\begin{align*}
P_++P_-+P_0=I_H
\end{align*}
and moreover the operator

\begin{align*}
L_\lambda=\lambda P_0+P_+-P_-
\end{align*}
is for each $\lambda\in [-1,1]$ a bounded selfadjoint Fredholm operator. More precisely, $L_\lambda\in GL(H)$ as long as $\lambda\neq 0$, and in the remaining case $L_0$ has a one dimensional kernel and cokernel which are both given by the span of $e_0$. Moreover, $L_\lambda$ is obviously a continuous path with respect to the norm topology and thus continuous with respect to the gap topology as well. Hence the spectral flow of $L$ is well defined. As

\begin{align*}
\sigma(L_\lambda)=\{-1,1,\lambda\},\quad\lambda\in [-1,1],
\end{align*}
it is immediate from the definition that

\begin{align*}
\sfl(L)=1.
\end{align*}
Moreover, we note that if we set $T_0=P_++P_0-P_-$ and $P=P_0$, then 

\begin{align*}
(I_H-P)T_0(I_H-P)&=(I_H-P_0)(P_++P_0-P_-)(I_H-P_0)=(I_H-P_0)(P_+-P_-)=P_+-P_-.
\end{align*}
In particular, $(I_H-P)T_0(I_H-P)$ defines a bounded, invertible and selfadjoint operator on $\ker P$ such that the path

\begin{align*}
\lambda P+(I_H-P)T_0(I_H-P)=L_\lambda,\quad\lambda\in[-1,1],
\end{align*} 
has spectral flow $1$.\\
Now we finally state the so called uniqueness of spectral flow which is the main result of \cite{LeschSpecFlowUniqu}. For a topological space $X$ and a subspace $Y\subset X$, we use the common notation $\Omega(X,Y)$ for the set of all paths in $X$ having endpoints in $Y$.

\begin{theorem}\label{IndPre-theorem-specflowuniqu}
Let $H$ be a separable Hilbert space and let

\begin{align*}
\mu:\Omega(\mathcal{CF}^{sa}(H),G\mathcal{C}^{sa}(H))\rightarrow\mathbb{Z}
\end{align*}
be a map which is additive with respect to concatenation of paths, invariant under gap continuous homotopies inside $\Omega(\mathcal{CF}^{sa}(H),G\mathcal{C}^{sa}(H))$ and which satisfies the following normalisation condition:\\
There is a rank one orthogonal projection $P\in\mathcal{L}(H)$ and a bounded selfadjoint operator $T_0$ having $\sigma(T_0)=\{-1,1\}$ such that the selfadjoint operator $(I_H-P)T_0(I_H-P)$ is invertible on $\ker P$ and

\begin{align*}
\mu(L)=1,
\end{align*}
where $L_\lambda=\lambda P+(I_H-P)T_0(I_H-P)$, $\lambda\in [-1,1]$.\\ 
Then $\mu$ is the spectral flow.
\end{theorem}

\noindent
Note that we have verified in this section that the spectral flow indeed has all the properties mentioned in its uniqueness theorem. In the proof of Theorem \ref{IndPre-theorem-specflowuniqu} in \cite{LeschSpecFlowUniqu}, it is shown that any gap continuous path can be deformed into a path in a certain normal form in which the spectral flow can be computed by considering finite dimensional matrices. Then the uniqueness of spectral flow follows from a corresponding result in finite dimensions which can also be found in \cite{LeschSpecFlowUniqu}.


\section{Crossing Forms}
In this section we briefly discuss a method which is often helpful for computing the spectral flow and which was introduced in \cite{Robbin-Salamon} and \cite{SFLPejsachowicz}, respectively, and recently generalised by the author in \cite{Homoclinics}. To this aim we have to restrict our considerations to special paths in $\mathcal{CF}^\textup{sa}(H)$.\\
Let $W\subset H$ be a Hilbert space in its own right and assume that $W\hookrightarrow H$ is continuous (e.g., $W=H$, or $W=H^1[0,1]$, $H=L^2[0,1]$). We denote by $\mathcal{BF}^\textup{sa}(W,H)$ the set of all $T\in\mathcal{CF}^\textup{sa}(H)$ such that $\mathcal{D}(T)=W$ and $T\in\mathcal{L}(W,H)$. In what follows we consider $\mathcal{BF}^\textup{sa}(W,H)$ as a topological subspace of the Banach space $\mathcal{L}(W,H)$. It is shown in \cite[Prop. 2.2]{LeschSpecFlowUniqu} that the inclusion

\[\mathcal{BF}^\textup{sa}(W,H)\hookrightarrow\mathcal{CF}^\textup{sa}(H)\] 
is continuous but not a topological embedding, i.e. the resulting topology on $\mathcal{BF}^\textup{sa}(W,H)$ is strictly stronger than the gap topology. From now on we assume that $\mathcal{A}:I\rightarrow\mathcal{BF}^\textup{sa}(W,H)$ is a continuously differentiable path.

\begin{defi}
An instant $\lambda_0\in[0,1]$ is called a \textit{crossing} if $\ker(\mathcal{A}_{\lambda_0})\neq 0$. The \textit{crossing form} at a crossing $\lambda_0$ is the quadratic form defined by

\[\Gamma(\mathcal{A},\lambda_0):\ker(\mathcal{A}_{\lambda_0})\rightarrow\mathbb{R},\,\,\Gamma(\mathcal{A},\lambda_0)[u]=\langle(\frac{d}{d\lambda}\mid_{\lambda=\lambda_0}\mathcal{A}_\lambda)u,u\rangle_H.\]
A crossing $\lambda_0$ is called \textit{regular}, if $\Gamma(\mathcal{A},\lambda_0)$ is non-degenerate.
\end{defi} 

\noindent
We mention without proof the following two theorems.

\begin{theorem}\label{thm:Sard}
There exists $\varepsilon>0$ such that

\begin{itemize}
	\item[i)] $\mathcal{A}+\delta\, I_H$ is a path in $\mathcal{BF}^\textup{sa}(W,H)$ for all $|\delta|<\varepsilon$;
	\item[ii)] $\mathcal{A}+\delta\, I_H$ has only regular crossings for almost every $\delta\in(-\varepsilon,\varepsilon)$.
\end{itemize}
\end{theorem}

\noindent
For a quadratic form $q:V\rightarrow\mathbb{R}$ on a finite dimensional Hilbert space $V$, there is a unique Hermitian matrix $A$ such that $q(u)=\langle Au,u\rangle$ for all $u\in V$. In what follows, we denote by $m^-(q)$ the number of negative eigenvalues of $A$ counted with multiplicities, and we set $m^+(q):=m^-(-A)$ as well as $\sgn(q):=m^+(q)-m^-(q)$.\\
The following theorem shows that the spectral flow of a continuously differentiable path $\mathcal{A}$ in $\mathcal{BF}^\textup{sa}(W,H)$ can be computed easily if all crossings are regular. 

\begin{theorem}\label{thm:crossing}
If $\mathcal{A}$ has only regular crossings, then they are finite in number and

\begin{align}\label{sflcross}
\sfl(\mathcal{A})=-m^-(\Gamma(\mathcal{A},0))+\sum_{\lambda\in(0,1)}{\sgn\Gamma(\mathcal{A},\lambda)}+m^+(\Gamma(\mathcal{A},1)).
\end{align}
\end{theorem}

\noindent
Theorem \ref{thm:Sard} and Theorem \ref{thm:crossing} were firstly proved by Robbin and Salamon in \cite{Robbin-Salamon} under the additional assumption that $W$ is compactly embedded in $H$. Later Fitzpatrick, Pejsachowicz and Recht proved Theorem \ref{thm:crossing} for bounded operators, i.e. if $W=H$. Note that the latter case is not covered by Robbin and Salamon's theorem as the inclusion $W\hookrightarrow H$ is not compact if $W$ and $H$ are of infinite dimension. A proof of both Theorems in the generality stated above can be found in \cite{Homoclinics}. Finally, let us mention that crossing forms can also be defined for paths of operators having varying domains \cite[App.A]{Wehrheim}.


\chapter{A Simple Example and a Glimpse at the Literature} 

\section{A Simple Example}

In this section we consider the Hilbert space $H=L^2[0,1]$ and the differential operators $\mathcal{A}_\lambda u=iu'$ on the domains

\[\mathcal{D}(\mathcal{A}_\lambda)=\{u\in H^1[0,1]:\, u(0)=e^{i\lambda}u(1)\}.\]
Our aim is to show that $\{\mathcal{A}_\lambda\}_{\lambda\in[-\pi,\pi]}$ is a continuous path in $\mathcal{CF}^\textup{sa}(H)$ and we want to compute its spectral flow.

\subsubsection*{The operators $\mathcal{A}_\lambda$ are selfadjoint and Fredholm}
We have already seen in Example \ref{example:selfadjointFredholm} that $\mathcal{A}_\lambda$ is selfadjoint for all $\lambda$. Moreover, we have shown that $\mathcal{A}_\lambda$ is surjective if $\lambda\neq 0$. As the kernel of $\mathcal{A}_\lambda$ is trivial in this case, we see that $\mathcal{A}_\lambda\in\mathcal{CF}^\textup{sa}(H)$ for $\lambda\neq 0$. Let us now consider the operator $\mathcal{A}_0$. Clearly, $\ker(\mathcal{A}_0)$ consists of all constant functions on $[0,1]$ and, moreover, it is readily seen that the image of $\mathcal{A}_0$ is 

\[\im(\mathcal{A}_0)=\{v\in L^2[0,1]:\,\int^1_0{v(t)\,dt}=0\},\]  
which is closed. Indeed, if $\{v_n\}_{n\in\mathbb{N}}$ is a sequence in $\im(\mathcal{A}_0)$ converging to some $v\in H$, then 

\[\left|\int^1_0{v(t)\,dt}\right|=\left|\int^1_0{v_n(t)-v(t)\,dt}\right|\leq\|v_n-v\|_H\rightarrow 0,\quad n\rightarrow\infty,\]
and so $v\in\im(\mathcal{A})$. Hence $\mathcal{A}_0\in\mathcal{CF}^\textup{sa}(H)$ by Lemma \ref{selfadjointperp}.\\
Let us note that there is another way to check that $\mathcal{A}_0$ is Fredholm. If we let $Y$ denote the subspace of $H$ consisting of all constant functions, then we have a direct sum decomposition $H=\im(\mathcal{A}_0)\oplus Y$. Indeed, the intersection of these spaces is trivial and, moreover, every function $v\in H$ can be written as 

\[v(t)=\underbrace{\left(v(t)-\int^1_0{v(t)\,dt}\right)}_{\in\im(\mathcal{A}_0)}+\underbrace{\int^1_0{v(t)\,dt}}_{\in Y}.\]
As the kernel of $\mathcal{A}_0$ is of finite dimension, this shows that $\mathcal{A}_{0}$ is Fredholm by Lemma \ref{codimclosed}.

\subsubsection*{The path $\{\mathcal{A}_\lambda\}_{\lambda\in[-\pi,\pi]}$ is continuous}

We define $T_\lambda:H\rightarrow H$ by $T_\lambda =(\mathcal{A}_\lambda+i)^{-1}$ and note that by \eqref{equivmetric} we need to show that $T_\lambda$ is a continuous path of bounded operators on $H$. To this aim, we will first compute $T_\lambda$ explicitly. If $v\in H$, then $u:=T_\lambda v$ is a solution of $iu'+iu=v$, and we obtain from standard methods for ordinary differential equations that 

\[u(t)=ce^{-t}-ie^{-t}\int^t_0{v(s)e^s\,ds},\quad t\in[0,1],\]
for some constant $c$ depending on $\lambda$ and $v$. As $u\in\mathcal{D}(\mathcal{A}_\lambda)$, we have to require that

\[c=u(0)=e^{i\lambda}u(1)=e^{i\lambda}\left(c e^{-1}-ie^{-1}\int^1_0{v(s) e^s\,ds} \right),\]
which implies that

\[c=-\frac{i}{e^{1-i\lambda}-1}\int^1_0{v(s) e^s\,ds}=:m(\lambda)\,\int^1_0{v(s) e^s\,ds},\]
and so we finally obtain

\[(T_\lambda v)(t)=m(\lambda)\,e^{-t}\int^1_0{v(s) e^s\,ds}-ie^{-t}\int^t_0{v(s) e^s\,ds}.\]
Clearly, $m:[-\pi,\pi]\rightarrow\mathbb{C}$ is a continuous function. If now $\lambda_0,\lambda_1\in[-\pi,\pi]$, then it is readily seen that

\begin{align*}
\|T_{\lambda_1} v-T_{\lambda_0}v\|_H\leq\frac{1}{\sqrt{2}}\sqrt{\cosh(2)-1}\,|m(\lambda_1)-m(\lambda_0)|\,\|v\|_H.
\end{align*}
Consequently, $\{T_\lambda\}_{\lambda\in[-\pi,\pi]}$ is a continuous path of bounded operators on $H=L^2[0,1]$ showing the continuity of $\mathcal{A}$ in $\mathcal{CF}^\textup{sa}(H)$.

\subsubsection*{The spectral flow of $\{\mathcal{A}_\lambda\}_{\lambda\in[-\pi,\pi]}$}
In order to compute the spectra of the operators $\mathcal{A}_\lambda$, we need to consider the differential equations

\[iu'-\mu u=0,\]
where $\mu$ is a real number. Clearly, the solution of this equation is given by $u(t)=c e^{-i\mu t}$, $t\in[0,1]$, for some constant $c$. As $u$ is an element of $\mathcal{D}(\mathcal{A}_\lambda)$, we get in addition

\[c=u(0)=e^{i\lambda}u(1)=c e^{i(\lambda-\mu)},\]
and so $\lambda-\mu=2k\pi$ for some $k\in\mathbb{Z}$. Consequently, we see that

\[\sigma(\mathcal{A}_\lambda)=\{2k\pi+\lambda:k\in\mathbb{Z}\}\]
and each element in $\sigma(\mathcal{A}_\lambda)$ is a simple eigenvalue, i.e. the corresponding eigenspace is one dimensional.\\
For the computation of the spectral flow, we set $\lambda_0=-\pi,\lambda_1=-\frac{\pi}{4}$, $\lambda_2=\frac{\pi}{4}$ and $\lambda_3=\pi$, as well as $a_1=\frac{\pi}{8}$, $a_2=\frac{\pi}{2}$ and $a_3=\frac{\pi}{8}$. Then $\dim E_{[0,a_i]}(\mathcal{A}_{\lambda_i})\neq 0$ only if $i=2$, and moreover $\dim E_{[0,a_i]}(\mathcal{A}_{\lambda_{i-1}})=0$ for all $i$. Hence

\[\sfl(\mathcal{A})=1.\]
Note that the path $\mathcal{A}$ is closed and hence we have found a non-trivial element in $\pi_1(\mathcal{CF}^\textup{sa}(H))$. This shows in particular that the spectral flow does not only depend on the endpoints of a path, for otherwise $\sfl(\mathcal{A})=0$ for every closed path $\mathcal{A}$.


\section{A Glimpse at the Literature}
In this final section we provide some literature on the spectral flow and its applications. Let us emphasise that our selection of articles is highly subjective and not at all an attempt for an exhaustive overview of the existing literature. It reflects the author's personal interests and we apologise to everyone not mentioned in the following paragraphs.

\subsection*{Definition of the Spectral Flow}
The spectral flow was introduced by Atiyah, Patodi and Singer in \cite{AtiyahPatodi} for paths in $\mathcal{BF}^\textup{sa}(H)$, and consequently also on $\mathcal{CF}^\textup{sa}(H)$ with respect to the Riesz metric (cf. Section \ref{section-topology}). A more analytic approach was given by Floer in \cite{Floer} and later fully developed by Phillips in \cite{Phillips}. Alternative constructions can be found in \cite{Robbin-Salamon}, for paths in $\mathcal{BF}^\textup{sa}(W,H)$, and in \cite{SFLPejsachowicz} for paths in $\mathcal{BF}^\textup{sa}(H)$. That the spectral flow can even be defined for paths of unbounded selfadjoint Fredholm operators which are merely continuous in the gap topology was observed by Booss-Bavnbek, Lesch and Phillips in \cite{UnbSpecFlow}. Finally, let us mention that Wahl introduced in \cite{Charlotte} a topology on $\mathcal{CF}^\textup{sa}(H)$ which is even weaker than the gap topology and she extended the spectral flow to this setting.

\subsection*{Uniqueness of the Spectral Flow}
As discussed in Section \ref{section-properties}, one may ask which properties characterise the spectral flow. Theorem \ref{IndPre-theorem-specflowuniqu} was proved by Lesch in \cite{LeschSpecFlowUniqu} and the same source also discusses uniqueness for $\mathcal{CF}^\textup{sa}(H)$ with respect to the Riesz metric, for $\mathcal{BF}^\textup{sa}(W,H)$ and for $\mathcal{BF}^\textup{sa}(H)$. An alternative approach for paths in $\mathcal{BF}^\textup{sa}(H)$ can be found in \cite{PejsaUniqueness}. Robbin and Salamon proved in \cite{Robbin-Salamon} uniqueness for $\mathcal{BF}^\textup{sa}(W,H)$ if $W$ is compactly embedded in $H$.

\subsection*{Crossing Forms and Partial Signatures}
Crossing forms and partial signatures are convenient tools to compute spectral flows. Crossing forms were introduced by Robbin and Salamon in \cite{Robbin-Salamon}, and later adapted in \cite{SFLPejsachowicz} and \cite{Homoclinics}. For partial signatures we refer to \cite{PiccionePartial}, and their application in \cite{PiccioneNicolaescu}.

\subsection*{Global Analysis}
Beginning from Atiyah, Patodi and Singer's article \cite{AtiyahPatodi}, probably most applications of the spectral flow were obtained in global analysis and related fields. We are definitely unable to give an exhaustive list of literature, and so we just want to mention \cite{BoossDesuspension}, \cite{Getzler}, \cite{Bunke}, \cite{Twisted} and that many applications deal with boundary value problems for Dirac operators, e.g. \cite{BunkeII}, \cite{NicolaescuII}, \cite{LeschWoj}, \cite{BoossSpectral}, \cite{Kirk}, \cite{PiccionePartial}, \cite{Prokhorova} and \cite{GoroLesch}. Some of these works equate the spectral flow in this setting with Maslov indices for curves of Lagrangian subspaces in infinite dimensional symplectic spaces (cf. \cite{BoossZhu}). Let us finally mention the author's contributions in \cite{SpinorsIch} and with Bei in \cite{Francesco}.

\subsection*{Symplectic Analysis}
Floer used the spectral flow in connection with his celebrated homology groups in \cite{Floer}. As in global analysis, there are too many papers dealing with the spectral flow in symplectic analysis to give an exhaustive bibliography. Let us just mention \cite{Taubes}, \cite{Yoshida}, \cite{Robbin-Salamon}, \cite{Seidel}, \cite{Wehrheim} and \cite{Paternain}. Also Kronheimer and Mrowka's monograph \cite{Mrowka} contains a section on the spectral flow, to which we refer for further literature.

\subsection*{Mathematical Physics}
Vafa and Witten used the spectral flow in \cite{VaWi} to obtain bounds on eigenvalues of Dirac operators (cf. also \cite{Atiyah}). More recent applications can be found, e.g., in \cite{BoossLeschEspo}, \cite{Pushnitzki}, \cite{Yuri}, \cite{Katsnelson}, \cite{Hermann} and \cite{StrohmaierII}, among many others.

\subsection*{Higher Spectral Flows}
The spectral flow was generalised to several-parameter families under additional assumptions by Dai and Zhang in \cite{DaiZhangI} and \cite{DaiZhangII}, where it is no longer an integer but a $K$-theory class of the parameter space. A central element in this construction are \textit{spectral sections}, which were introduced by Melrose and Piazza in \cite{Paolo}.

\subsection*{Spectral Flow and Operator Algebras}
There is a vast literature on different extensions of the construction of the spectral flow to more general settings. As we are hardly acquainted with these topics, we just want to quote without further comments  \cite{Perera}, \cite{PhillipsII}, \cite{Carey} \cite{Leichtnam}, \cite{CharlotteIII}, \cite{CharlotteI}, \cite{SaraCharlotte} and \cite{Kaad}.

\subsection*{Nonlinear Analysis}
Fitzpatrick, Pejsachowicz and Recht discovered in \cite{SFLPejsachowicz} that the spectral flow can also be used in the bifurcation theory of critical points of strongly indefinite functionals on Hilbert spaces, for which they introduced an alternative construction of the spectral flow. Their result was later improved in \cite{BifJac} by Pejsachowicz in a joint work with the author. Recently, Alexander and Fitzpatrick revisited the spectral flow in bifurcation theory in \cite{Alexander}.

\subsection*{2nd order Partial Differential Equations}
Formulas for the spectral flow for paths of second order PDEs can be found, e.g., in the recent work \cite{Goffeng} by Goffeng and Schrohe, as well as the author's papers \cite{AleIch} with Portaluri and \cite{CompIch}.

\subsection*{Hamiltonian Systems}
The bifurcation theory developed in \cite{SFLPejsachowicz} was firstly applied to bifurcation of periodic orbits of Hamiltonian systems in \cite{SFLPejsachowiczII}. Further applications of the spectral flow in case of periodic orbits can be found in \cite{Robbin-Salamon}, \cite{Izydorek}, \cite{BifJac} and \cite{CalcVar}. The spectral flow for homoclinic solutions was considered in \cite{Hu}, \cite{Jacobo} and \cite{Homoclinics}. Also the works \cite{Pejsachowicz}, \cite{Zhu}, \cite{MussoPejsachowicz} and \cite{MorseK} on a generalisation of the Morse index theorem to geodesics in semi-Riemannian manifolds may be considered as an application to Hamiltonian Systems.

\vspace{1cm}
Nils Waterstraat\\
School of Mathematics,\\
Statistics \& Actuarial Science\\
University of Kent\\
Canterbury\\
Kent CT2 7NF\\
UNITED KINGDOM\\
E-mail: n.waterstraat@kent.ac.uk

\end{document}